\documentclass[11pt, a4paper, oneside, reqno]{amsart}

\usepackage[utf8]{inputenc}
\usepackage{amsmath, amsthm, amssymb, amsfonts}
\usepackage{bbm}
\usepackage[foot]{amsaddr}
\usepackage{booktabs}
\usepackage[shortlabels]{enumitem}
\usepackage[colorlinks=true, linkcolor = blue, citecolor = dgreen]{hyperref}
\usepackage{cleveref}
\usepackage[numbers]{natbib}
\usepackage{mathrsfs}
\usepackage[dvipsnames]{xcolor}
\usepackage{tikz}
\usepackage{geometry}
\usepackage{setspace}
\usepackage[normalem]{ulem}
\usepackage{newpxtext}
\usepackage{thmtools}
\setlist{leftmargin=*, topsep=0.5em, parsep=0pt, itemsep=1em, labelindent=0pt, align=left}

\definecolor{red}{HTML}{D62728}
\definecolor{blue}{RGB}{ 0, 109, 219}
\definecolor{dgreen}{rgb}{0,.8,0}

\crefname{enumi}{}{}
\Crefname{enumi}{}{}

\crefformat{enumi}{(#2#1#3)}
\Crefformat{enumi}{(#2#1#3)}
\crefrangeformat{enumi}{(#3#1#4) to~(#5#2#6)}
\Crefrangeformat{enumi}{(#3#1#4) to~(#5#2#6)}

\makeatletter
\newcommand{\ie}{\emph{i.e.}\@ifnextchar.{\!\@gobble}{}}
\newcommand{\eg}{\emph{e.g.}\@ifnextchar.{\!\@gobble}{}}
\newcommand{\etc}{etc\@ifnextchar.{}{.\@}}
\makeatother

\newtheorem*{theorem*}{Theorem}
\newtheorem{thm}{Theorem}[section]
\newtheorem{dfn}[thm]{Definition}
\newtheorem{lemma}[thm]{Lemma}
\newtheorem{prop}[thm]{Proposition}

\newtheorem{remark}[thm]{Remark}

\numberwithin{equation}{section}

\crefname{thm}{theorem}{theorems}
\Crefname{thm}{Theorem}{Theorems}
\crefname{lemma}{lemma}{lemmas}
\Crefname{lemma}{Lemma}{Lemmas}
\crefname{prop}{proposition}{propositions}
\Crefname{prop}{Proposition}{Propositions}
\crefname{corollary}{corollary}{corollaries}
\Crefname{corollary}{Corollary}{Corollaries}
\crefname{dfn}{definition}{definitions}
\Crefname{dfn}{Definition}{Definitions}
\crefname{remark}{remark}{remarks}
\Crefname{remark}{Remark}{Remarks}
\crefname{example}{example}{examples}
\Crefname{example}{Example}{Examples}
\crefname{conjecture}{conjecture}{conjectures}
\Crefname{conjecture}{Conjecture}{Conjectures}
\crefname{hypothesis}{hypothesis}{hypotheses}
\Crefname{hypothesis}{Hypothesis}{Hypotheses}
\crefname{problem}{problem}{problems}
\Crefname{problem}{Problem}{Problems}

\def\E{{\mathbb E}} 

\def \R {{\mathbb{R}}}

\newcommand{\independent}{\perp\!\!\!\perp}

\newcommand{\sgn}{\text{sgn}}

\AddToHook{cmd/appendix/before}{%
  \crefalias{chapter}{appendix}%
  \crefalias{section}{appendix}%
  \crefalias{subsection}{appendix}%
}

\begin{document}
\title[]{Exponential integrability and limiting behavior of the derivative of intersection and self-intersection local time of fractional Brownian motion}

\author{Kaustav Das$^{\dagger \ddagger}$}
\author{Gregory Markowsky$^{\dagger}$}
\author{Binghao Wu$^{\dagger}$}
\author{Qian Yu$^{*}$}
\address{$^\dagger$School of Mathematics, Monash University, Victoria, 3800 Australia.}
\address{$^\ddagger$Centre for Quantitative Finance and Investment Strategies, Monash University, Victoria, 3800 Australia.}
\address{$^*$School of Mathematics, Nanjing University of Aeronautics and Astronautics,
Nanjing 211106, P. R. China.}
\email{kaustav.das@monash.edu, greg.markowsky@monash.edu, binghao.wu@monash.edu, qyumath@163.com}
\date{}

\begin{abstract}
We give the correct condition for existence of the $k$-th derivative of the intersection local time for fractional Brownian motion,  which was originally discussed in [Guo, J., Hu, Y., and Xiao, Y., Higher-order derivative of intersection local time for two independent fractional Brownian motions, Journal of Theoretical Probability 32, (2019), pp. 1190-1201]. We also show that the $k$-th derivative of the intersection and self-intersection local times of fractional Brownian motion are exponentially integrable for certain parameter values. In addition, we show convergence in distribution when the existence condition is violated for the $k$-th derivative of self-intersection local time of fractional Brownian motion under scaling. 
\end{abstract}
\maketitle
\vspace{.4cm}
\noindent Keywords: Brownian motion; local time; self-intersection local time; derivatives of self-intersection local time; fractional Brownian motion; exponential integrability, Malliavin calculus, Central limit theorem.
\section{Introduction and main results}
\label{sec:introduction}
\noindent

Let $B_t$ be a Brownian motion for the time being, and consider the following functional introduced in \citep{rogers1991intrinsic, rogers1991local, rogers1991t},
\begin{align*}
    A(T,B_{T})=\int_{0}^{T}1_{[0,\infty)}(B_T-B_s)ds.
\end{align*}
A formal application of It\^o's formula, using $\frac{d}{dx}1_{[0,\infty)}(x) = \delta(x), \frac{d^2}{dx^2}1_{[0,\infty)}(x) = \delta'(x)$ with $\delta$ the Dirac delta function, leads to a Tanaka-style formula containing the following expression:
\begin{align} 
	\int_{0}^T\int_{0}^t\delta'(B_{t}-B_{s})dsdt, \label{eqn:DSLTBM}
\end{align}
This motivated the influential work \citep{rosen2005derivatives}, where existence of this process, known as the derivative of self-intersection local time (DSLT) of Brownian motion, was rigorously proved, and a number of properties of the process provided. The corresponding Tanaka formula was also stated as a formal identity in that paper, although later (\citep{markowsky2008proof}) the following slightly different formula was rigorously proved:

\begin{align*}
    \frac{1}{2}\int_{0}^T\int_{0}^t\delta'(B_{t}-B_{s})dsdt + \frac{1}{2} \sgn(x)T =
    \int_0^T L_s^{B_s - x} dB_s - \int_{0}^{T} \sgn(B_{T}-B_{s} - x) ds. 
\end{align*} 

Since that time a lengthy sequence of papers devoted to DSLT by many authors have followed, many of which have focused on the same expression for processes other than Brownian motion. We will continue that study in this paper, with our interest being DSLT of fractional Brownian motion (fBm).

In what follows, $B^{H}$ will denote a one-dimensional fBm with Hurst parameter $H$. The DSLT of fBm was first introduced by Yan, Yang, and Lu in \citep{yan2008p}; however, as was noted in that paper, there are two natural versions of the DSLT of fBm. The first version is derived from the Tanaka formula, and was justified by Jung and Markowsky \citep{jung2014tanaka}. They showed that when the Hurst parameter $0<H<\tfrac{2}{3}$, the DSLT of fractional Brownian motion
\[
    -H\int_{0}^{t}\int_{0}^{s} \delta^{\prime}(B_{s}^{H}-B_{r}^{H})(s-r)^{2H-1} \, dr \, ds
\]
exists in $L^{p}(\Omega)$, where $B^{H}$ denotes a one-dimensional fractional Brownian motion. Note that the kernel $(s-r)^{2H-1}$ is present due to the form of Ito's formula for fBm.

The second version is derived from the occupation time formula and was also proven to exist under the same condition on the Hurst parameter by Jung and Markowsky \citep{jung2015holder}. Specifically, when $0<H<\tfrac{2}{3}$,
\[
    \int_{0}^{t}\int_{0}^{s} \delta^{\prime}(B_{s}^{H}-B_{r}^{H}) \, dr \, ds
\]
exists in $L^{p}(\Omega)$. In this article, we will work only with this second version, i.e. without the kernel $(s-r)^{2H-1}$.

Inspired by the work above, Yu \citep{yu2021higher} showed that for $d$-dimensional fractional Brownian motion $B^{H}$, the $k$-th order DSLT
\[
    (-1)^{|k|}\int_{0}^{t}\int_{0}^{s}\delta^{(k)}(B^{H}_{s}-B^{H}_{r}) \, dr \, ds
\]
exists in $L^{2}(\Omega)$ when
\[
    H < \min\!\left(\tfrac{2}{2|k|+d}, \; \tfrac{1}{|k|+d-\#}, \; \tfrac{1}{d}\right),
\]
and exists in $L^{p}(\Omega)$ when
\[
    H|k| + Hd < 1,
\]
where $k=(k_{1},\dots,k_{d}) \in \mathbb{N}^{d}$, $|k|=\sum_{j=1}^{d}k_{j}$, and $\#$ denotes the number of odd $k_{i}$ in $k$.  For convenience, we neglect the constant term and denote the following as the DSLT of fractional Brownian motion:
\begin{align}
\label{DSLT of fbm}
    \hat{\alpha}_{t}^{(k)} := \int_{D} \delta^{(k)}(B_{s}^{H}-B^{H}_{r}) \, dr \, ds
    = \lim_{\epsilon \to 0} \hat{\alpha}^{(k)}_{t,\epsilon}
    := \lim_{\epsilon \to 0} \int_{D} \delta^{(k)}_{\epsilon}(B_{s}^{H}-B_{r}^{H}) \, dr \, ds ,
\end{align}
where $D=\{(r,s) \,|\, 0<r<s<t\}$. Following Yu's work, a number of subsequent papers have studied these processes more closely; see \cite{kuang2022derivative,das2022existence,yu2020asymptotic,guo2024derivative,yu2023smoothness,hong2025exact,yu2024limit}.

Another focus of this paper is the derivative of intersection local time (DILT) of fractional Brownian motion, which is formally defined as
\begin{align}
\label{DILT of fbm}
    \alpha_{t}^{(k)} := \int_{0}^{t}\int_{0}^{t} \delta^{(k)}(B^{H}_{s}-\hat{B}^{H}_{r}) \, dr \, ds,
\end{align}
where $B^{H}$ and $\hat{B}^{H}$ are two independent $d$-dimensional fractional Brownian motions with the same Hurst parameter $H$. Recall that a $d$-dimensional fractional Brownian motion with Hurst parameter $H \in (0,1)$, denoted by $B^H$, is a $d$-dimensional centered Gaussian process, continuous a.s., whose $d$ components are independent copies of a one-dimensional fractional Brownian motion $B^{H,j}$, $j \in \{1,\dots,d\}$, with covariance function
\[
    \mathbb{E}[B^{H,j}_{t} B^{H,j}_{s}] = \tfrac{1}{2} \big( t^{2H} + s^{2H} - |t-s|^{2H} \big).
\]
Note that when $H=\tfrac{1}{2}$, fractional Brownian motion reduces to standard Brownian motion. Other than this case, the increments of a fractional Brownian motion are not independent. Naturally, in order to rigorously define $\alpha$ and $\hat \alpha$, one must begin with an approximate $\delta$ function and then show convergence to a well defined process. To be precise, we let
\begin{align*}
    \delta_{\epsilon}(x) :=\frac{1}{(2\pi \epsilon)^{\frac{d}{2}}}e^{-\frac{|x|^{2}}{2\epsilon}}
\end{align*}
be our approximate $\delta$ function, and it can be shown that $\delta_\epsilon$ converges weakly to $\delta$ as $\epsilon \to 0$. We then utilise the representation of $\delta_{\epsilon}$ through the Fourier transform
\begin{align*}
	\delta_{\epsilon}(x) &= \frac{1}{(2\pi)^{d}} \int_{\R^{d}}e^{i\langle p,x\rangle}e^{-\frac{\epsilon|p|^{2}}{2}}dp, \\
    \delta^{(k)}_{\epsilon}(x)=& \frac{i^{|k|}}{(2\pi)^{d}}\int_{\R^{d}}\prod_{j=1}^{d}p^{k_{j}}_{j}e^{i\langle p,x\rangle}e^{-\frac{\epsilon|p|^{2}}{2}}dp,
\end{align*} 
where $\langle \cdot , \cdot \rangle$ denotes the $d$-dimensional Euclidean inner product. For simplicity, we focus on $t=1$, and denote
\begin{align}
\label{limit of DILT}
    \alpha^{(k)} := \int_{0}^{1}\int_{0}^{1} \delta^{(k)}(B^{H}_{s}-\hat{B}^{H}_{r}) \, dr \, ds,
\end{align}
and define the mollified version
\begin{align}
\label{molifiled version}
    \alpha_{\epsilon}^{(k)} &:= \int_{0}^{1}\int_{0}^{1} \delta_{\epsilon}^{(k)}(B^{H}_{s}-\hat{B}^{H}_{r}) \, dr \, ds,\nonumber\\
    &=\frac{i^{|k|}}{(2\pi)^{d}}\int_{0}^{1}\int_{0}^{1}\int_{\R^{d}}\prod_{j=1}^{d}p^{k_{j}}_{j}e^{ip(B_{s}^{H}-\hat{B}_{r}^{H})}e^{\frac{-\epsilon|p|^{2}}{2}}dpdrds,
\end{align}
where other cases can be obtained by scaling.
The existence of $\alpha^{(k)}$ in $L^{p}(\Omega)$ was discussed in \citep{guo2019higher}; however, unfortunately, an error was noted in their proof, and a counterexample to their result was discussed in \citep{das2025exponential}. Our first order of business will therefore be to give the correct range of existence for the process, which we do in our first result.

 \begin{thm}
\label{theorem of existence of DILT of fbm}
Let $k=(k_{1},\dots,k_{d}) \in \mathbb{N}^{d}$ and $|k|=k_{1}+\cdots+k_{d}$. Suppose that $2|k|H+Hd<2$. Then $\alpha_{\epsilon}^{(k)}$ defined in \cref{molifiled version} converges in $L^{n}(\Omega)$ as $\epsilon\to 0$ for all $n \geq 1$. 
\end{thm}
\begin{remark}

We denote its limit by $\alpha^{(k)}$ which is defined in \cref{limit of DILT}. When $|k|=0$, the existence condition reduces to $Hd<2$, which coincides with the critical existence condition in \citep{nualart2007intersection}.
\end{remark}

Having established existence, we turn to the question of exponential integrability. We will say that a random variable $X$ is {\it exponentially integrable} of order $\beta$ if there exists a constant $M > 0$ such that $\E[\exp \{ M|X|^{\beta} \}]<\infty$. Exponential integrability with $\beta=1$ is equivalent to the existence of the moment generating function $M_X(t) := \E[e^{tX}]$ for $t$ in a neighborhood of $0$, and can also be used to give strong tail estimates on the distribution of $X$. Exponential integrability of various flavors of intersection local time has been an important concept in mathematical physics, particular in relation to models of self attracting or self avoiding Brownian motion and polymer measures. More details can be found in the influential works \citep{bass2004self,konig2006brownian,le1994exponential}. In \cite{das2025exponential}, exponential integrability was provided for DILT and DSLT of Brownian motion and $\alpha$-stable processes, and we will extend these results to fBm, as follows.

\begin{thm}
\label{theorem of EI of DILT of fbm}
Let $k=(k_{1},\dots,k_{d}) \in \mathbb{N}^{d}$ and $|k|=k_{1}+\cdots+k_{d}$. Suppose that
\[
    2|k|H+Hd<2 
    \quad \text{and} \quad 
    \beta<\frac{1}{|k|+|k|H+dH}.
\]
Then there exists a constant $M>0$ such that
    \[
        \mathbb{E}\!\left[e^{M(\alpha^{(k)})^{\beta}}\right]<\infty.
    \]
\end{thm}

We note that there are also studies of the DILT of two fractional Brownian motions with different Hurst parameters (see \citep{guo2019higher,hong2020derivatives,hong2021derivatives,yan2017derivative}), but in this article we will only consider the case where the two fractional Brownian motions have the same Hurst parameter. The property of local nondeterminism of Gaussian processes, first proposed by Berman \citep{berman1973local}, plays an important role in analyzing the moments of DSLT. Local nondeterminism asserts that one cannot accurately anticipate the value of a Gaussian process at a point, no matter how close the available information is to that point. For the case of fractional Brownian motion, we will make use of the version of local nondeterminism established by Hu and Nualart \citep{hu2008integral}, which states that for $t,s,r \in [0,T]$, there exists a positive constant $\kappa$ depending only on $H$ and $T$ such that
\begin{align}
    \label{local nondeterminism}
    \mathrm{Var}\!\left(B_{t}^{H} \,\middle|\, B_{s}^{H}, |t-s|>r \right) > \kappa r^{2H}.
\end{align}

We will provide a similar result for $\hat{\alpha}$. For simplicity, we focus on $t=1$, and denote
\begin{align*}
    \hat{\alpha}^{(k)} := \int_{0}^{1}\int_{0}^{s} \delta^{(k)}(B^{H}_{s}-B^{H}_{r}) \, dr \, ds,
\end{align*}
since other cases can be obtained by scaling.
\begin{thm} \label{theorem of EI of DSLT of fbm}
    When $H|k|+Hd<1$ and $\beta < \frac{1}{|k|+|k|H+dH}$, there exists a constant $M>0$ such that 
    \begin{align*}
        \mathbb{E}[e^{M(\hat{\alpha}^{(k)})^{\beta}}]<\infty,
    \end{align*}
    where $k=(k_{1},\dots,k_{d})\in \mathbb{N}^{d}$ and $|k|=k_{1}+\cdots+k_{d}$.
\end{thm}

It is also of interest to study limiting behavior of the mollified processes in the cases which do not converge. The idea here seems to have originated in \citep{yor1985renormalisation}, which in turn was influenced by Varadhan's celebrated renormalization of self-intersection local time of planar Brownian motion \citet{varadhan1969appendix}. This has led to a large number of similar results, too many to list here; however, even in relation to DSLT of fractional Brownian motion, we can refer the reader to \citep{rosen1988limit, markowsky2008renormalization, rosen1992renormalization, bock2014polymer, jaramillo2017asymptotic, jaramillo2019functional, xu2024central, yu2020asymptotic, yu2024limit}

We will prove similar theorems for the DSLT of fBm, as follows. 

\begin{thm}
\label{thm1}
    If $d=2$ and $\frac{1}{2}<H<1$, then
    \begin{align*}
        \lim_{\epsilon\to 0}\epsilon^{2-\frac{1}{H}}\hat{\alpha}^{(1)}_{t,\epsilon}\overset{d}{=} \mathcal{N}(0,\sigma_{1}^{2}),
    \end{align*}
    where $\sigma_{1}^{2}=\frac{(2H-1)t^{2H}}{H8\pi^{2}}B(2,2H-1)B(\frac{1}{H},\frac{4H-2}{2H})^{2}$, and $B(\cdot,\cdot)$ is the Beta function.
\end{thm}
\begin{thm}
\label{thm2}
    If $d=3$ and $\frac{1}{2}<H<\frac{2}{3}$, then
    \begin{align*}
        \lim_{\epsilon\to 0}\epsilon^{\frac{5}{2}-\frac{1}{H}}\hat{\alpha}_{t,\epsilon}^{(1)}\overset{d}{=} \mathcal{N}(0,\sigma_{2}^{2})
    \end{align*}
    where $\sigma_{2}^{2}=\frac{(2H-1)t^{2H}}{H16\pi^{3}}B(2,2H-1)B(\frac{1}{H},\frac{5H-2}{2H})^{2}$, and $B(\cdot,\cdot)$ is the Beta function.
\end{thm}

In the ensuing two sections we will prove these theorems, however they will require a significant amount of preliminaries and technical lemmas. We have placed this material into two appendices at the end of this article: \Cref{malliavin appendix} focuses on Malliavin calculus, whereas \Cref{appen:misc} is simply a collection of miscellaneous methods that are needed.

\section{Proofs of existence and exponential integrability, Theorems \ref{theorem of existence of DILT of fbm}, \ref{theorem of EI of DILT of fbm}, and \ref{theorem of EI of DSLT of fbm}.}

The heart of the matter is the following proposition, which provides the key estimate for \Cref{theorem of existence of DILT of fbm,theorem of EI of DILT of fbm}.

\begin{prop}
\label{even moment bound}
Let $n$ be even. If  $2|k|H+Hd<2$, then there exists a constant 
$C_{H,d,|k|}>0$, depending only on $H,d,$ and $|k|$, such that
\[
    I_{1} := \int_{[0,1]^{2n}} \int_{\mathbb{R}^{nd}}
    \exp\!\left(-\tfrac{1}{2}\sum_{j=1}^{d}\xi_{j}^{\intercal}A\xi_{j}\right)
    \prod_{j=1}^{d}\prod_{l=1}^{n}|\xi_{lj}|^{k_{j}} \, d\xi \, ds \, dr
    \;\;\leq\;\; C_{H,d,|k|}^{\,n}(n!)^{\,|k|+|k|H+Hd},
\]
where $A$ is the covariance matrix of the random vector
\[
    \big(B^{H,1}_{s_{1}}-\hat{B}^{H,1}_{r_{1}}, \dots, B^{H,1}_{s_{n}}-\hat{B}^{H,1}_{r_{n}}\big),
\]
and $\xi_{j}=(\xi_{1j},\dots,\xi_{nj})^{\intercal}\in \R^{n}$.
\end{prop}
\begin{proof}
    Clearly, $A$ is symmetric and positive definite. Hence there exists a symmetric positive definite matrix $B=(b_{ij})_{1\leq i,j\leq n}$ such that $B^{2}=A^{-1}$.
    Note that 
    \begin{align}
    \label{product inequality}
    \begin{split}
        \prod_{j=1}^{d}|\xi_{lj}|^{k_{j}}&\leq \prod_{j=1}^{d}|\sum_{m=1}^{d}\xi_{lm}^{2}|^{\frac{k_{j}}{2}}=\left|\sum_{m=1}^{d}\xi_{lm}^{2}\right|^{\frac{|k|}{2}}.
    \end{split}
    \end{align}

As such, we obtain
\begin{align*}
I_{1}  &\leq \frac{1}{(2\pi)^{nd}}\int_{[0,1]^{2n}}\int_{\R^{nd}}e^{-\frac{1}{2}\sum_{j=1}^{d}\xi_{j}^{^\intercal}A\xi_{j}}\prod_{l=1}^{n}\left|\sum_{m=1}^{d}\xi_{lm}^{2}\right|^{\frac{|k|}{2}}d\xi dsdr.
\end{align*}
Now we change variables $A^{\frac{1}{2}}\xi_{j}=u_{j}$, 
which yields
\begin{align*}
      I_{1}&\leq \frac{1}{(2\pi)^{nd}}\int_{[0,1]^{2n}}\det(A)^{-\frac{1}{2}d}\int_{\R^{nd}}e^{-\frac{1}{2}\sum_{j=1}^{d}u_{j}^{\intercal}u_{j}}\prod_{l=1}^{n}\left|\sum_{m=1}^{d}\left(\sum_{j=1}^{n}b_{lj}u_{jm}\right)^{2}\right|^{\frac{|k|}{2}}dudsdr\\
       &\leq  \frac{1}{(2\pi)^{nd}}\int_{[0,1]^{2n}}\det(A)^{-\frac{1}{2}d}\int_{\R^{nd}}e^{-\frac{1}{2}\sum_{j=1}^{d}u_{j}^{\intercal}u_{j}}\prod_{l=1}^{n}\left|\sum_{m=1}^{d}\sum_{j=1}^{n}b_{lj}^{2}\sum_{j=1}^{n}u_{jm}^{2}\right|^{\frac{|k|}{2}}dudsdr,
\end{align*}
where we apply the Cauchy--Schwarz inequality in the second inequality.  
It is well known that
\[
    A^{-1} = \frac{1}{\det(A)}C^{\intercal},
\]
where $C=(c_{ij})_{1\leq i,j\leq n}$ is the cofactor matrix of $A$. In particular, $c_{ll}=\det(A_{l})$, where $A_{l}$ is the submatrix of $A$ obtained by deleting its $l$-th row and $l$-th column. Note that 
\begin{align*}
    \sum_{j=1}^{n}b_{lj}^{2}&=\sum_{j=1}^{n}b_{lj}b_{jl} =A^{-1}_{ll} =\frac{c_{ll}}{\det(A)}. 
\end{align*} 
According to \Cref{lemma:Gaussiandeterminant}, for any permutation $\sigma$ of $\{1,\dots,n\}$ and $\pi$ of $\{1,\dots,n\}\setminus \{l\}$, 
\begin{align*}
    \det(A)&=\text{Var}(B^{H,1}_{s_{\sigma(1)}}-\hat{B}_{r_{\sigma(1)}}^{H,1})\times \text{Var}(B^{H,1}_{s_{\sigma(2)}}-\hat{B}_{r_{\sigma(2)}}^{H,1}|B^{H,1}_{s_{\sigma(1)}}-\hat{B}_{r_{\sigma(1)}}^{H,1})\times\cdots\\&\quad\times \text{Var}(B^{H,1}_{s_{\sigma(n)}}-\hat{B}_{r_{\sigma(n)}}^{H,1}|B^{H,1}_{s_{\sigma(j)}}-\hat{B}_{r_{\sigma(j)}}^{H,1},1\leq j \leq n-1),\\
    c_{ll}&=\det(A_{l})=\prod_{j\neq l}^{n}\text{Var}(B_{s_{\pi(j)}}^{H,1}-\hat{B}_{r_{\pi(j)}}^{H,1}|B^{H,1}_{s_{\pi(i)}}-\hat{B}^{H,1}_{r_{\pi(i)}}, i\in\{1,\dots,j-1\}\setminus\{l\}).
\end{align*}
Therefore we can choose proper $\sigma$ and $\pi$ such that 
\begin{align*}
\det(A)/c_{ll}=\text{Var}(B_{s_{l}}^{H,1}-\hat{B}^{H,1}_{r_{l}}|B_{s_{p}}^{H,1}-\hat{B}^{H,1}_{r_{p}}, p\in \{1,\dots,n\}\setminus \{l\}).
\end{align*}
Hence, we obtain
\begin{align*}
 I_{1}&\leq \frac{1}{(2\pi)^{nd}}\int_{[0,1]^{2n}}\det(A)^{-\frac{1}{2}d}\prod_{l=1}^{n}\left(\frac{c_{ll}}{\det(A)}\right)^{\frac{|k|}{2}}\times \int_{\R^{nd}}e^{-\frac{1}{2}\sum_{j=1}^{d}u_{j}^{\intercal}u_{j}}\left|\sum_{m=1}^{d}\sum_{j=1}^{n}u^{2}_{jm}\right|^{\frac{n|k|}{2}}du dsdr\\
     &\leq \frac{1}{(2\pi)^{nd}}\int_{[0,1]^{2n}}\det(A)^{-\frac{1}{2}d}\prod_{l=1}^{n}\left(\text{Var}(B_{s_{l}}^{H,1}-\hat{B}^{H,1}_{r_{l}}|B_{s_{p}}^{H,1}-\hat{B}^{H,1}_{r_{p}},1\leq p\neq l\leq n)\right)^{-\frac{|k|}{2}}\\
    &\quad\times \int_{\R^{nd}}e^{-\frac{1}{2}\sum_{j=1}^{d}u_{j}^{\intercal}u_{j}}\left|\sum_{m=1}^{d}\sum_{j=1}^{n}u^{2}_{jm}\right|^{\frac{n|k|}{2}}dudsdr.
\end{align*}
Due to the independence between $B^{H}$ and $\hat{B}^{H}$ and \Cref{Conditional Variance inequality}, we have
\begin{align}
    \det(A)&=\text{Var}(B^{H,1}_{s_{\sigma(1)}}-\hat{B}_{r_{\sigma(1)}}^{H,1})\times \text{Var}(B^{H,1}_{s_{\sigma(2)}}-\hat{B}_{r_{\sigma(2)}}^{H,1}|B^{H,1}_{s_{\sigma(1)}}-\hat{B}_{r_{\sigma(1)}}^{H,1})\times\cdots \nonumber \\
    &\quad\times \text{Var}(B^{H,1}_{s_{\sigma(n)}}-\hat{B}_{r_{\sigma(n)}}^{H,1}|B^{H,1}_{s_{\sigma(j)}}-\hat{B}_{r_{\sigma(j)}}^{H,1},1\leq j \leq n-1) \nonumber \\
    &\geq 2^{-n}\prod_{j=1}^{n}(\text{Var}(B_{s_{\sigma(j)}}^{H,1}|B_{s_{\sigma(k)}}^{H,1},1\leq k\leq j-1, \hat{B}_{r_{\sigma(i)}}^{H,1},1\leq i\leq j)\label{eqn:ineq1}\\
    &\quad+\text{Var}(\hat{B}_{r_{\sigma(j)}}^{H,1}|B_{s_{\sigma(k)}}^{H,1},1\leq k\leq j, \hat{B}_{r_{\sigma(i)}}^{H,1},1\leq i\leq j-1)) \nonumber\\
    &\geq \prod_{j=1}^{n}\text{Var}(B_{s_{\sigma(j)}}^{H,1}|B^{H,1}_{s_{\sigma(m)}},1\leq m\leq j-1)^{\frac{1}{2}}\text{Var}(\hat{B}_{r_{\sigma(j)}}^{H,1}|\hat{B}^{H,1}_{r_{\sigma(m)}},1\leq m\leq j-1)^{\frac{1}{2}} \label{eqn:ineq2}
\end{align}
for any permutation $\sigma$. We now give brief explanation for the two inequalities above. For the first inequality \cref{eqn:ineq1} we have used \Cref{Conditional Variance inequality} as follows
\begin{align*}
   &\text{Var}(B_{s_{\sigma(j)}}^{H,1}-\hat{B}^{H,1}_{r_{\sigma(j)}}|B_{s_{\sigma(i)}}^{H,1}-\hat{B}_{r_{\sigma(i)}}^{H,1},1\leq i\leq j-1)\geq \\
    \max(&\text{Var}(B_{s_{\sigma(j)}}^{H,1}-\hat{B}_{r_{\sigma(j)}}^{H,1}|B_{s_{\sigma(k)}}^{H,1},1\leq k\leq j-1, \hat{B}_{r_{\sigma(i)}}^{H,1},1\leq i\leq j)\\
    , &\text{Var}(B_{s_{\sigma(j)}}^{H,1}-\hat{B}_{r_{\sigma(j)}}^{H,1}|B_{s_{\sigma(k)}}^{H,1},1\leq k\leq j, \hat{B}_{r_{\sigma(i)}}^{H,1},1\leq i\leq j-1)),
\end{align*}
which implies
\begin{align*}
    2\text{Var}(B_{s_{\sigma(j)}}^{H,1}-\hat{B}^{H,1}_{r_{\sigma(j)}}|B_{s_{\sigma(i)}}^{H,1}-\hat{B}_{r_{\sigma(i)}}^{H,1},1\leq i\leq j-1)&\geq \text{Var}(B_{s_{\sigma(j)}}^{H,1}|B_{s_{\sigma(k)}}^{H,1},1\leq k\leq j-1, \hat{B}_{r_{\sigma(i)}}^{H,1},1\leq i\leq j)\\
    &\quad+\text{Var}(\hat{B}_{r_{\sigma(j)}}^{H,1}|B_{s_{\sigma(k)}}^{H,1},1\leq k\leq j, \hat{B}_{r_{\sigma(i)}}^{H,1},1\leq i\leq j-1)
\end{align*}
for $2\leq j\leq n$. Note that 
\begin{align*}
    \text{Var}(B_{s_{\sigma(1)}}^{H,1}-\hat{B}_{r_{\sigma(1)}}^{H,1})&= \text{Var}(B_{s_{\sigma(1)}}^{H,1})+\text{Var}(\hat{B}_{r_{\sigma(1)}}^{H,1})\\
    &\geq\frac{1}{2}(\text{Var}(B_{s_{\sigma(1)}}^{H,1})+\text{Var}(\hat{B}_{r_{\sigma(1)}}^{H,1})),
\end{align*}
and the first inequality \cref{eqn:ineq1} follows. For the second inequality \cref{eqn:ineq2}, since the natural filtrations of $B^{H,1}$ and $\hat{B}^{H,1}$ are independent, the $\sigma$-algebra generated by $B_{s_{\sigma(1)}}^{H,1},\dots,B_{s_{\sigma(j)}}^{H,1}$ is independent of the one generated by $\hat{B}^{H,1}_{r_{\sigma(1)}},\dots,\hat{B}_{r_{\sigma(j)}}^{H,1}$ for all $1\leq j\leq n$. Hence by \Cref{conditional expectation extra independent sigma} together with  $a+b\geq2\sqrt{ab}$ for all $a,b\geq 0$, we have
\begin{align*}
    &\text{Var}(B_{s_{\sigma(j)}}^{H,1}|B_{s_{\sigma(k)}}^{H,1},1\leq k\leq j-1, \hat{B}_{r_{\sigma(i)}}^{H,1},1\leq i\leq j)+\text{Var}(\hat{B}_{r_{\sigma(j)}}^{H,1}|B_{s_{\sigma(k)}}^{H,1},1\leq k\leq j, \hat{B}_{r_{\sigma(i)}}^{H,1},1\leq i\leq j-1)\\
     &\geq2\sqrt{\text{Var}(B_{s_{\sigma(j)}}^{H,1}|B_{s_{\sigma(i)}}^{H,1},1\leq i\leq j-1)\text{Var}(\hat{B}_{r_{\sigma(j)}}^{H,1}|\hat{B}_{r_{\sigma(i)}}^{H,1},1\leq i\leq j-1)}
\end{align*}
for all $1\leq j\leq n$. As such, the second inequality \cref{eqn:ineq2} follows.
Denote by $\Phi_{n}$ the set of all permutations of $\{1,\dots,n\}$, and $\Delta_{\sigma}^{n}=\{(s_{\sigma(1)},\dots,s_{\sigma(n)})\in [0,1]^{n};0\leq s_{\sigma(1)}\leq\cdots \leq s_{\sigma(n)}\leq 1\}$.
As such, we have
\begin{align*}
     I_{1}&\leq \frac{1}{(2\pi)^{nd}}\sum_{\sigma,\pi\in\Phi_{n}}\int_{\Delta_{\sigma}^{n}\times\Delta_{\pi}^{n}}\prod_{j=1}^{n}\left(\text{Var}(B_{s_{\sigma(j)}}^{H,1}|B^{H,1}_{s_{\sigma(m)}},1\leq m\leq j-1)\right)^{-\frac{d}{4}}\\
     &\quad\times\left(\text{Var}(\hat{B}_{r_{\pi(j)}}^{H,1}|\hat{B}^{H,1}_{r_{\pi(m)}},1\leq m\leq j-1)\right)^{-\frac{d}{4}}\times \prod_{l=1}^{n}\left(\text{Var}(B_{s_{l}}^{H,1}|B_{s_{p}}^{H,1},1\leq p\neq l\leq n)\right)^{-\frac{|k|}{4}}\\
     &\quad\times \left(\text{Var}(\hat{B}_{r_{l}}^{H,1}|\hat{B}_{r_{p}}^{H,1},1\leq p\neq l\leq n)\right)^{-\frac{|k|}{4}}\times \int_{\R^{nd}}e^{-\frac{1}{2}\sum_{j=1}^{d}u_{j}^{\intercal}u_{j}}\left|\sum_{m=1}^{d}\sum_{j=1}^{n}u^{2}_{jm}\right|^{\frac{n|k|}{2}}dudsdr\\
    &=(n!)^{2}\frac{1}{(2\pi)^{nd}}\Biggl(\int_{\{0\leq s_{1}\cdots\leq s_{n}\leq1\}}\prod_{j=1}^{n}\text{Var}(B_{s_{j}}^{H,1}|B^{H,1}_{s_{m}},1\leq m\leq j-1)^{-\frac{d}{4}}\\
     &\quad\times \prod_{l=1}^{n}\text{Var}(B_{s_{l}}^{H,1}|B_{s_{p}}^{H,1},1\leq p\neq l\leq n)^{-\frac{|k|}{4}}ds\Biggl)^{2}\times \int_{\R^{nd}}e^{-\frac{1}{2}\sum_{j=1}^{d}u_{j}^{\intercal}u_{j}}\left|\sum_{m=1}^{d}\sum_{j=1}^{n}u^{2}_{jm}\right|^{\frac{n|k|}{2}}du.
\end{align*}
Denote
\begin{align*}
    \Lambda&:=\int_{\{0\leq s_{1}\cdots\leq s_{n}\leq1\}}\prod_{j=1}^{n}\text{Var}(B_{s_{j}}^{H,1}|B^{H,1}_{s_{m}},1\leq m\leq j-1)^{-\frac{d}{4}}\\
     &\quad\times \prod_{l=1}^{n}\text{Var}(B_{s_{l}}^{H,1}|B_{s_{p}}^{H,1},1\leq p\neq l\leq n)^{-\frac{|k|}{4}}ds, \\
     \Omega&:=\int_{\R^{nd}}e^{-\frac{1}{2}\sum_{j=1}^{d}u_{j}^{\intercal}u_{j}}\left|\sum_{m=1}^{d}\sum_{j=1}^{n}u^{2}_{jm}\right|^{\frac{n|k|}{2}}du.
\end{align*}
According to the local nondeterminism \cref{local nondeterminism},
\begin{align}
\label{Bounds for det(A)}
\prod_{j=1}^{n}\text{Var}(B_{s_{j}}^{H,1}|B^{H,1}_{S_{m}},1\leq m\leq j-1)^{-\frac{d}{4}}&\leq \kappa^{-\frac{dn}{4}}\prod_{j=1}^{n}(s_{j}-s_{j-1})^{-\frac{Hd}{2}},
\end{align}
\begin{align}
\label{Bounds for prod Var}
\begin{split}
\prod_{l=1}^{n}\text{Var}(B_{s_{l}}^{H,1}|B_{s_{j}}^{H,1},1\leq j\neq l\leq n)^{-\frac{|k|}{4}}&\leq \kappa^{-\frac{|k|n}{4}}(s_{n}-s_{n-1})^{-\frac{H|k|}{2}}(s_{2}-s_{1})^{-\frac{H|k|}{2}}\\
    &\quad\times\prod_{l=2}^{n-1}\min((s_{l}-s_{l-1})^{2H},(s_{l+1}-s_{l})^{2H})^{-\frac{|k|}{4}}\\
    &\leq \kappa^{-\frac{|k|n}{2}}(s_{n}-s_{n-1})^{-\frac{H|k|}{2}} (s_{2}-s_{1})^{-\frac{H|k|}{2}}\\
    &\quad\times\prod_{l=2}^{n-1}((s_{l}-s_{l-1})^{-\frac{H|k|}{2}}+(s_{l+1}-s_{l})^{-\frac{H|k|}{2}}),
    \end{split}
\end{align}
on the set $\{0\leq s_{1}\leq \cdots\leq s_{n}\leq 1\}$.
Combining \cref{Bounds for det(A)} and \cref{Bounds for prod Var}, we have 
\begin{align*}
    \Lambda&\leq \kappa^{-\frac{|k|n+dn}{4}}\int_{\{0\leq s_{1}\leq\cdots\leq s_{n\leq 1}\}}(s_{2}-s_{1})^{-\frac{H|k|}{2}}(s_{n}-s_{n-1})^{-\frac{H|k|}{2}}\prod_{j=1}^{n}(s_{j}-s_{j-1})^{-\frac{Hd}{2}}\\
    &\quad\times\prod_{l=2}^{n-1}((s_{l}-s_{l-1})^{-\frac{H|k|}{2}}+(s_{l+1}-s_{l})^{-\frac{H|k|}{2}})ds\\
    &=\kappa^{-\frac{|k|n+dn}{4}}\sum_{J\in2^{\{2,\dots,n-1\}}}\int_{\{0\leq s_{1}\leq\cdots\leq s_{n\leq 1}\}}(s_{2}-s_{1})^{-\frac{H|k|}{2}}(s_{n}-s_{n-1})^{-\frac{H|k|}{2}}\prod_{j=1}^{n}(s_{j}-s_{j-1})^{-\frac{Hd}{2}}\\
    &\quad\times\prod_{l\in J}(s_{l}-s_{l-1})^{-\frac{H|k|}{2}}\prod_{l\in J^{c}}(s_{l+1}-s_{l})^{-\frac{H|k|}{2}}ds\\
    &\leq \kappa^{-\frac{|k|n+nd}{4}}\sum_{J\subset\{2,\dots,n-1\}} \frac{c^{n}}{\Gamma(n(1-\frac{H|k|}{2}-\frac{Hd}{2})+1)}\\
    &\leq \kappa^{-\frac{|k|n+nd}{4}}2^{n-2}c^{n}(n!)^{\frac{Hd}{2}+\frac{H|k|}{2}-1}\\
    &\leq C_{H,d,|k|}^{n}(n!)^{\frac{Hd}{2}+\frac{H|k|}{2}-1}.
\end{align*}
Note that the integrand in the third line contains the term 
\[
    (s_{2}-s_{1})^{-H|k|-\tfrac{Hd}{2}} \, (s_{n}-s_{n-1})^{-H|k|-\tfrac{Hd}{2}}
\]
when $J$ includes both $2$ and $n-1$. This term is the most dominant one.  
Therefore, we apply \Cref{Beta function} in the second inequality under the condition $2H|k|+Hd<2$, and then use \Cref{lmag} in the third inequality. Here $C_{H,d,|k|}$ denotes a positive constant depending only on $H$, $d$, and $|k|$.
Note that 
\begin{align}
\label{Omega}
\begin{split}
    \Omega&=\int_{\R^{nd}}e^{-\frac{1}{2}\sum_{j=1}^{d}v_{j}^{\intercal}v_{j}}\left|\sum_{m=1}^{d}\sum_{j=1}^{n}v^{2}_{jm}\right|^{\frac{n|k|}{2}}dv\\
    &\leq \int_{\R^{nd}} e^{-\frac{1}{2}\sum_{j=1}^{d}\sum_{m=1}^{n}v_{mj}^{2}}(nd)^{\frac{n|k|}{2}}\left|\max_{1\leq m\leq n, 1\leq j\leq d}v_{mj}^{n|k|}\right|dv\\
    &\leq (nd)^{\frac{n|k|}{2}}\int_{\R^{nd}}e^{-\frac{1}{2}\sum_{j=1}^{d}\sum_{m=1}^{n}v_{mj}^{2}}\left|\sum_{m=1}^{n}\sum_{j=1}^{d}v_{mj}^{n|k|}\right|dv\\
    &\leq(nd)^{\frac{n|k|}{2}}\int_{\R^{nd}}e^{-\frac{1}{2}\sum_{j=1}^{d}\sum_{m=1}^{n}v_{mj}^{2}}\sum_{m=1}^{n}\sum_{j=1}^{d}\left|v_{mj}^{n|k|}\right|dv\\
    &=(nd)^{\frac{n|k|}{2}+1}\int_{\R^{nd}}e^{-\frac{1}{2}\sum_{j=1}^{d}\sum_{m=1}^{n}v_{mj}^{2}}\left|v_{11}^{n|k|}\right|dv\\
    &\leq C_{d,|k|}^{n}n^{\frac{n|k|}{2}}\int_{\R}e^{-\frac{1}{2}v_{11}^{2}}\left|v_{11}^{n|k|}\right|dv_{11}\\
    &\leq C_{d,|k|}^{n}n^{\frac{n|k|}{2}}(n|k|-1)!!\\
    &\leq C_{d,|k|}^{n}(n!)^{|k|},
\end{split}
\end{align}
where we use Stirling’s estimate in the last inequality, and $C_{d,|k|}$ denotes a positive constant depending only on $d$ and $|k|$ (whose value may change from line to line). Hence,
\begin{align}
\label{Bounds for moments of DILT of fbm}
\begin{split}
I_{1}&\leq 2^{-n}(n!)^{2}\Lambda^{2}\Omega\\
    &\leq C_{H,d,|k|}^{n}(n!)^{|k|+|k|H+Hd}.
\end{split}
\end{align}
\end{proof}
\begin{proof}[Proof of \Cref{theorem of existence of DILT of fbm,theorem of EI of DILT of fbm}]
   According to \Cref{condition for convergence in L^{p}} and \Cref{even L^{p}}, it suffices to show that
\[
    \mathbb{E}\!\left[(\alpha_{\epsilon_{1}}^{(k)})^{q} \, (\alpha_{\epsilon_{2}}^{(k)})^{\,n-q}\right]
\]
converges to the same value as $\epsilon_{1},\epsilon_{2}\to 0$, for all even $n$ and all $1\leq q\leq n$. Assuming $\epsilon_{1},\epsilon_{2}>0$ and  according to Fubini's theorem, we have
    \begin{align*}
        \mathbb{E}[(\alpha^{(k)}_{\epsilon_{1}})^{q}(\alpha^{(k)}_{\epsilon_{2}})^{n-q}]
        &=\frac{(i)^{n|k|d}}{(2\pi)^{nd}}\int_{[0,1]^{2n}}\int_{\R^{nd}}\mathbb{E}[e^{i\sum_{j=1}^{d}\sum_{l=1}^{n}\xi_{lj}(B^{H,j}_{s}-\hat{B}_{r}^{H,j})}](e^{-\frac{\epsilon_{1}\sum_{j=1}^{d}\sum_{l=1}^{q}\xi_{lj}^{2}}{2}}e^{-\frac{\epsilon_{2}\sum_{j=1}^{d}\sum_{l=q+1}^{n}\xi_{lj}^{2}}{2}})\\ &\quad\times\prod_{j=1}^{d}\prod_{l=1}^{n}\xi_{lj}^{k_{j}}d\xi drds\\
        &=\frac{(i)^{n|k|d}}{(2\pi)^{nd}}\int_{[0,1]^{2n}}\int_{\R^{nd}}e^{-\frac{1}{2}\sum_{j=1}^{d}\xi_{j}^{^\intercal}A\xi_{j}}(e^{-\frac{\epsilon_{1}\sum_{j=1}^{d}\sum_{l=1}^{q}\xi_{lj}^{2}}{2}}e^{-\frac{\epsilon_{2}\sum_{j=1}^{d}\sum_{l=q+1}^{n}\xi_{lj}^{2}}{2}})\\
        &\quad\times\prod_{j=1}^{d}\prod_{l=1}^{n}\xi_{lj}^{k_{j}}d\xi drds,
    \end{align*}
     where $A$ is the covariance matrix of the random vector $(B^{H,1}_{s_{1}}-\hat{B}^{H,1}_{r_{1}},\dots,B_{s_{n}}^{H,1}-\hat{B}_{r_{n}}^{H,1})$ and $\xi_{j}=(\xi_{1j},\dots,\xi_{nj})^{\intercal}$. Note that 
     \begin{align*}
        e^{-\frac{1}{2}\sum_{j=1}^{d}\xi_{j}^{^\intercal}A\xi_{j}}(e^{-\frac{\epsilon_{1}\sum_{j=1}^{d}\sum_{l=1}^{q}\xi_{lj}^{2}}{2}}e^{-\frac{\epsilon_{2}\sum_{j=1}^{d}\sum_{l=q+1}^{n}\xi_{lj}^{2}}{2}})\prod_{j=1}^{d}\prod_{l=1}^{n}\xi_{lj}^{k_{j}} \longrightarrow        e^{-\frac{1}{2}\sum_{j=1}^{d}\xi_{j}^{^\intercal}A\xi_{j}}\prod_{j=1}^{d}\prod_{l=1}^{n}\xi_{lj}^{k_{j}} 
     \end{align*}
     as $\epsilon_{1},\epsilon_{2}\to 0$ and it is bounded by 
     \begin{align*}
         e^{-\frac{1}{2}\sum_{j=1}^{d}\xi_{j}^{^\intercal}A\xi_{j}}\prod_{j=1}^{d}\prod_{l=1}^{n}|\xi_{lj}|^{k_{j}}.
     \end{align*}
     By the dominated convergence theorem and \Cref{even moment bound},
     \begin{align*}
         \mathbb{E}[(\alpha^{(k)}_{\epsilon_{1}})^{q}(\alpha^{(k)}_{\epsilon_{2}})^{n-q}] \longrightarrow \frac{(i)^{n|k|d}}{(2\pi)^{nd}}\int_{[0,1]^{2n}}\int_{\R^{nd}}e^{-\frac{1}{2}\sum_{j=1}^{d}\xi_{j}^{^\intercal}A\xi_{j}}\prod_{j=1}^{d}\prod_{l=1}^{n}\xi_{lj}^{k_{j}}d\xi dsdr
     \end{align*}
     as $\epsilon_{1},\epsilon_{2}\to0$ for all even $n$ and $1\leq q\leq n$ when $2|k|H+Hd<2$.
     This implies that $\alpha_{\epsilon}^{(k)}$ converges in all $L^{n}(\Omega)$ as $\epsilon\to 0$ by \Cref{even L^{p}} and \Cref{condition for convergence in L^{p}}. We therefore can denote its limit by $\alpha^{(k)}$ and its even $n$-th moment is as follows:
     \begin{align*}
         \mathbb{E}[(\alpha^{(k)})^{n}]&=\mathbb{E}\left[\frac{(i)^{n|k|d}}{(2\pi)^{nd}}\int_{[0,1]^{2n}}\int_{\R^{nd}}e^{i\sum_{j=1}^{d}\sum_{l=1}^{n}\xi_{lj}(B^{H,j}_{s}-\hat{B}_{r}^{H,j})}\prod_{j=1}^{d}\prod_{l=1}^{n}\xi_{lj}^{k_{j}}d\xi dsdr\right]\\
         &\leq \frac{1}{(2\pi)^{nd}}\int_{[0,1]^{2n}}\int_{\R^{nd}}\mathbb{E}\left[e^{i\sum_{j=1}^{d}\sum_{l=1}^{n}\xi_{lj}(B^{H,j}_{s}-\hat{B}_{r}^{H,j})}\right]\prod_{j=1}^{d}\prod_{l=1}^{n}\xi_{lj}^{k_{j}}d\xi dsdr\\
         &\leq\frac{1}{(2\pi)^{nd}}\int_{[0,1]^{2n}}\int_{\R^{nd}} e^{-\frac{1}{2}\sum_{j=1}^{d}\xi_{j}^{^\intercal}A\xi_{j}}\prod_{j=1}^{d}\prod_{l=1}^{n}|\xi_{lj}|^{k_{j}}d\xi dsdr\\
         &\leq  C_{H,d,|k|}^{n}(n!)^{|k|+|k|H+Hd},
     \end{align*}
where the last inequality follows from \Cref{even moment bound}. The odd moments of $\alpha^{(k)}$ can be tackled by Jensen's inequality. Supposing $n$ is odd and utilising Jensen's inequality on the concave function $f(x)=x^{\frac{n}{n+1}}$ with the random variable $|\theta|^{n+1}$, we obtain
\begin{align*}
    \mathbb{E}[|\alpha^{(k)}_{\epsilon}|^{n}]&=\mathbb{E}[|\alpha^{(k)}_{\epsilon}|^{(n+1)\frac{n}{n+1}}]\\
    &\leq \mathbb{E}[|\alpha^{(k)}_{\epsilon}|^{(n+1)}]^{\frac{n}{n+1}}\\
    &\leq C_{H,d,|k|}^{n}((n+1)!)^{\frac{n(|k|+|k|H+dH)}{n+1}}\\
    &\leq C_{H,d,|k|}^{n}(n!)^{|k|+|k|H+dH},
\end{align*}
where we have obtained the preceding 2nd inequality via \Cref{even moment bound}. Regarding the 3rd inequality, we have applied \Cref{lmag} to $((n+1)!)^{\frac{n}{n+1}}(\frac{n}{n+1})^{n+1}$:
\begin{align*}
    ((n+1)!)^{\frac{n}{n+1}}\left (\frac{n}{n+1}\right )^{n+1}&\leq \Gamma\left (\left (\frac{n}{n+1}\right )(n+1)+1 \right)
\end{align*}
and thus
\begin{align*}
    ((n+1)!)^{\frac{n}{n+1}}&\leq (n!)\left (\frac{n+1}{n}\right )^{n+1}.
\end{align*}
 Hence for $0\leq\beta<\frac{1}{|k|+|k|H+Hd}$ and $n\in\mathbb{N}$, by Jensen's inequality, we have
\begin{align*}
    \mathbb{E}\left[|\alpha^{(k)}|^{\beta n}\right]\leq \mathbb{E}\left[|\alpha^{(k)}|^{n}\right]^\beta\leq C_{H,d,|k|}^{n\beta}(n!)^{\beta(|k|+|k|H+Hd)},
\end{align*}
Thus, by Monotone convergence theorem
\begin{align*}  
    \mathbb{E}\left[e^{M|\alpha|^{\beta}}\right] = \sum_{n=0}^{\infty}\frac{M^{n}\mathbb{E}[|\alpha^{(k)}|^{\beta n}]}{n!} 
    \leq \sum_{n=0}^{\infty}M^{n}C_{H,d,|k|}^{n}(n!)^{\beta(|k|+|k|H+Hd)-1} 
    < \infty
\end{align*}
for all $M>0$.
\end{proof}

\if2
For simplicity, we focus on $t=1$, and denote
\begin{align*}
    \hat{\alpha}^{(k)} := \int_{0}^{t}\int_{0}^{s} \delta^{(k)}(B^{H}_{s}-B^{H}_{r}) \, ds \, dr,
\end{align*}
since other cases can be obtained by scaling.
\begin{thm}
    When $H|k|+Hd<1$ and $\beta < \frac{1}{|k|+|k|H+dH}$, there exists a constant $M>0$ such that 
    \begin{align*}
        \mathbb{E}[e^{M(\hat{\alpha}^{(k)})^{\beta}}]<\infty,
    \end{align*}
    where $k=(k_{1},\dots,k_{d})\in \mathbb{N}^{d}$ and $|k|=k_{1}+\cdots+k_{d}$.
\end{thm}
\fi 

\begin{proof}[Proof of Theorem \ref{theorem of EI of DSLT of fbm} (sketch)]
Denote $A$ as the covariance matrix of the random vector $(B^{H,1}_{s_{1}}-B^{H,1}_{r_{1}},\dots,B_{s_{n}}^{H,1}-B_{r_{n}}^{H,1})$, where $B^{H,1}$ is the first component of $B^{H}$. Since $A$ is symmetric positive definite, there exists a matrix $(b_{ij})_{1\leq i,j\leq n}=B$ such that $A^{-1}=B^{2}$. Similar to the proof of \Cref{theorem of existence of DILT of fbm}, for any even $n$, we have
    \begin{align*}
        \mathbb{E}[|\hat{\alpha}^{(k)}|^{n}]\leq \frac{1}{(2\pi)^{nd}}\int_{[0,1]^{2n}}\int_{\R^{nd}}e^{-\frac{1}{2}\sum_{j=1}^{d}\xi_{j}^{^\intercal}A\xi_{j}}\prod_{j=1}^{d}\prod_{l=1}^{n}|\xi_{lj}|^{k_{j}}d\xi dsdr.
    \end{align*}
Note that we can apply the same technique as in \Cref{even moment bound} to obtain
\begin{align*}
    \mathbb{E}[|\hat{\alpha}^{(k)}|^{n}]
     &\leq \frac{1}{(2\pi)^{nd}}\int_{[0,1]^{2n}}\det(A)^{-\frac{1}{2}d}\prod_{l=1}^{n}\left(\text{Var}(B_{s_{l}}^{H,1}-B^{H,1}_{r_{l}}|B_{s_{p}}^{H,1}-B^{H,1}_{r_{p}},1\leq p\neq l\leq n)\right)^{-\frac{|k|}{2}}\\
    &\quad\times \int_{\R^{nd}}e^{-\frac{1}{2}\sum_{j=1}^{d}u_{j}^{\intercal}u_{j}}\left|\sum_{m=1}^{d}\sum_{j=1}^{n}u^{2}_{jm}\right|^{\frac{n|k|}{2}}dudsdr.
\end{align*}
Denote 
\begin{align*}
    \Lambda&:=\int_{[0,1]^{2n}}\det(A)^{-\frac{1}{2}d}\prod_{l=1}^{n}\left(\text{Var}(B_{s_{l}}^{H,1}-B^{H,1}_{r_{l}}|B_{s_{p}}^{H,1}-B^{H,1}_{r_{p}},1\leq p\neq l\leq n)\right)^{-\frac{|k|}{2}}dsdr,\\
    \Omega&:=\int_{\R^{nd}}e^{-\frac{1}{2}\sum_{j=1}^{d}u_{j}^{\intercal}u_{j}}\left|\sum_{m=1}^{d}\sum_{j=1}^{n}u^{2}_{jm}\right|^{\frac{n|k|}{2}}du.
\end{align*}
Let $\tau_{l}$ be the closest point from the left to $s_{l}$ taking value from $\{r_{l},r_{l+1},\dots,r_{n}, s_{l-1}\}$, and $\lambda_{l}$ be the closest point from the right to $s_{l}$  taking value from $\{r_{l+1},\dots,r_{n},s_{l+1}\}$ for all $1\leq l\leq n$.
Due to the local nondeterminism \cref{local nondeterminism} and \Cref{Conditional Variance inequality}, when $1\leq l\leq n-1$,
\begin{align*}
    \text{Var}(B_{s_{l}}^{H,1}-B^{H,1}_{r_{l}}|B_{s_{p}}^{H,1}-B^{H,1}_{r_{p}},1\leq p\neq l\leq n)&\geq \text{Var}(B_{s_{l}}^{H,1}|B_{s_{p}}^{H,1},1\leq p\neq l\leq n,B^{H,1}_{r_{m}},1\leq m\leq n)\\
    &\geq \kappa \min((s_{l}-\tau_{l})^{2H}, (\lambda_{l}-s_{l})^{2H}),
\end{align*}
when $l=n$,
\begin{align*}
     \text{Var}(B_{s_{n}}^{H,1}-B^{H,1}_{r_{n}}|B_{s_{p}}^{H,1}-B^{H,1}_{r_{p}},1\leq p < n)\geq \kappa (s_{n}-\tau_{n})^{2H}
\end{align*}
on the set $D^{n}\cap\Delta_{n}$, where $D=\{(r,s);0<r<s<1\}$ and $\Delta_{n}=\{(s_{1},\dots,s_{n})\in [0,1]^{n};0\leq s_{1}\leq\cdots \leq s_{n}\leq 1\}$.
Applying the same technique and \Cref{lemma:Gaussiandeterminant}, we obtain
\begin{align*}
    \det(A)&=\prod_{l=1}^{n}\text{Var}(B_{s_{l}}^{H,1}-B^{H,1}_{r_{l}}|B_{s_{p}}^{H,1}-B^{H,1}_{r_{p}},1\leq p\leq l-1)\\
    &\geq \prod_{l=1}^{n}\text{Var}(B_{s_{l}}^{H,1}|B_{s_{p}}^{H,1},1\leq p\leq l-1, B_{r_{m}}^{H,1},1\leq m\leq l)\\
    &\geq\kappa^{n}\prod_{l=1}^{n}(s_{l}-\tau_{l})^{2H}
\end{align*}
on $D^{n}\cap\Delta_{n}$.
Hence, 
\begin{align*}
    \Lambda&=n!\int_{D^{n}\cap\Delta_{n}}\det(A)^{-\frac{1}{2}d}\prod_{l=1}^{n}\left(\text{Var}(B_{s_{l}}^{H,1}-B^{H,1}_{r_{l}}|B_{s_{p}}^{H,1}-B^{H,1}_{r_{p}},1\leq p\neq l\leq n)\right)^{-\frac{|k|}{2}}drds\\
    &\leq \kappa^{\frac{-dn-d|k|}{2}}(n!)\int_{D^{n}\cap\Delta_{n}}(s_{n}-\tau_{n})^{-|k|H}\prod_{i=1}^{n}(s_{i}-\tau_{i})^{-dH}\prod_{j=1}^{n-1}(\frac{1}{\min((s_{j}-\tau_{j})^{2H},(\lambda_{j}-s_{j})^{2H})})^{\frac{|k|}{2}}drds\\
    &\leq C_{d,|k|,H}^{n}(n!)\int_{D^{n}\cap\Delta_{n}}(s_{n}-\tau_{n})^{-|k|H}\prod_{i=1}^{n}(s_{i}-\tau_{i})^{-dH}\prod_{j=1}^{n-1}((s_{j}-\tau_{j})^{-|k|H}+(\lambda_{j}-s_{j})^{-|k|H})drds,
\end{align*}
where the first equality follows from the symmetry of the integrand, and $C_{d,|k|,H}$ denotes a positive constant depending on $d$, $|k|$, and $H$, whose value may vary from line to line. Consider one configuration of $E=\{0<z_1<\cdots<z_{2n}\}\subset D^{n}\cap\Delta_{n}$, there must exist a mapping 
\begin{align*}
    \sigma : \{1,\dots,n\}\to \{1,\dots,2n\}
\end{align*}
such that $z_{\sigma_{(i)}}=s_{i}$. As such, when $H|k|+Hd<1$,
\begin{align*}
&\int_{E}(s_{n}-\tau_{n})^{-|k|H}\prod_{i=1}^{n}(s_{i}-\tau_{i})^{-dH}\prod_{j=1}^{n-1}((s_{j}-\tau_{j})^{-|k|H}+(\lambda_{j}-s_{j})^{-|k|H})drds\\
&=\int_{E}(z_{\sigma(n)}-z_{\sigma(n)-1})^{-|k|H}\prod_{i=1}^{n}(z_{\sigma(i)}-z_{\sigma(i)-1})^{-dH}\prod_{j=1}^{n-1}((z_{\sigma(j)}-z_{\sigma(j)-1})^{-|k|H}+(z_{\sigma(j)+1}-z_{\sigma(j)})^{-|k|H})dz\\
&=\int_{E}(z_{\sigma(n)}-z_{\sigma(n)-1})^{-|k|H}\prod_{i=1}^{n}(z_{\sigma(i)}-z_{\sigma(i)-1})^{-dH}\\
&\quad\times\sum_{J\in 2^{\{1,\dots,n-1\}}}\prod_{j\in J}(z_{\sigma(j)}-z_{\sigma(j)-1})^{-|k|H}\prod_{j\in J^{c}}(z_{\sigma(j)+1}-z_{\sigma(j)})^{-|k|H}dz\\
&\leq 2^{n-1}\frac{c^{2n}}{\Gamma(-|k|Hn-dHn+2n+1)}\\
&\leq C^{n}(n!)^{H|k|+Hd-2},
\end{align*}
where $C$ is a positive constant independent of $H$, $|k|$, and $d$.  
We have used \Cref{Beta function} in the second preceding inequality, and applied Stirling’s estimate together with \Cref{lmag} in the last inequality.  

As we can see, the bound does not depend on the choice of $E$.  
In fact, there are $(2n-1)!!$ possible choices of $E$, since the $r_{i}$'s can be placed sequentially as follows: $r_{1}$ can only be placed in $(0,s_{1})$; $r_{2}$ can be placed in $(0,r_{1})$, $(r_{1},s_{1})$, or $(s_{1},s_{2})$; and so on. Therefore, we have 
\begin{align*}
    \Lambda &\leq C_{d,|k|,H}^{n}(n!)((2n-1)!!)(n!)^{k|H|+dH-2}\\
    &\leq C_{d,|k|,H}^{n}(n!)^{|k|H+dH},
\end{align*}
where we have used $(2n-1)!!=\frac{2n!}{2^{n}n!}$ and Stirling estimate for it in the second inequality. According to \cref{Omega},
\begin{align*}
    \mathbb{E}[|\hat{\alpha}^{(k)}|^{n}] &\leq  \frac{1}{(2\pi)^{nd}}\Lambda\Omega\\
    &\leq C_{d,|k|,H}^{n}(n!)^{|k|+H|k|+Hd}.
\end{align*}
Using the same argument in the proof as in \Cref{theorem of EI of DILT of fbm}, when $n$ is odd,
\begin{align*}
     \mathbb{E}[|\hat{\alpha}^{(k)}|^{n}] \leq C_{d,|k|,H}^{n}(n!)^{|k|+H|k|+Hd}.
\end{align*}
 Hence for $0\leq\beta<\frac{1}{|k|+|k|H+Hd}$ and $n\in\mathbb{N}$, we have
\begin{align*}
    \mathbb{E}\left[|\hat{\alpha}^{(k)}|^{\beta n}\right]\leq \mathbb{E}\left[|\hat{\alpha}^{(k)}|^{n}\right]^\beta\leq C_{H,d,|k|}^{n\beta}(n!)^{\beta(|k|+|k|H+Hd)}.
\end{align*}
Thus,
\begin{align*}  
    \mathbb{E}\left[e^{M|\hat{\alpha}^{(k)}|^{\beta}}\right] = \sum_{n=0}^{\infty}\frac{M^{n}\mathbb{E}[|\hat{\alpha}^{(k)}|^{\beta n}]}{n!} 
    \leq \sum_{n=0}^{\infty}M^{n}C_{H,d,|k|}^{n}(n!)^{\beta(|k|+|k|H+Hd)-1} 
    < \infty
\end{align*}
for all $M>0$.
\end{proof}

\section{Proof of \Cref{thm1,thm2}}

We now proof the central limit theorems, \Cref{thm1,thm2}. We remind the reader that required preliminaries on Malliavin calculus can be found in \Cref{malliavin appendix} below. The proof will require a series of calculations, which we have organized into lemmas, before we begin the proof proper.

\begin{lemma}
\label{lemm1}
When $d=2$, $\frac{1}{2}<H<1$, we have
\begin{align*}
    \lim_{\epsilon\to 0}\mathbb{E}[\epsilon^{4-\frac{2}{H}}|\hat{\alpha}^{(1)}_{t,\epsilon}|^{2}]=\sigma_{1}^{2},
\end{align*}
where $\sigma_{1}^{2}$ is defined in \Cref{thm1}.
\end{lemma}
\begin{proof}
    According to Lemma 3.1 in \citep{yu2024limit}, 
    \begin{align*}
        \mathbb{E}[|\hat{\alpha}^{(1)}_{t,\epsilon}|^{2}]=V_{1}(\epsilon)+V_{2}(\epsilon)+V_{3}(\epsilon)
    \end{align*}
    with 
    \begin{align*}
        V_{i}(\epsilon)=\frac{2}{(2\pi)^{2}}\int_{D_{i}}|\epsilon I+\Sigma|^{-2}|\mu|drdsdr^{'}ds^{'},
    \end{align*}
    where $D_{i}$ defined in \Cref{Bounds} and $\Sigma$ is a covariance matrix with $\Sigma_{1,1}=\lambda$ , $\Sigma_{2,2}=\rho$ and $\Sigma_{1,2}=\mu$ given in \Cref{Bounds}. For the $V_{1}(\epsilon)$ term, changing variables $(r,r^{'},s,s^{'})$ by ($r$, $r^{'}-r=a$, $s-r^{'}=b$, $s^{'}-s=c$), there exists some $C>0$ such that
    \begin{align*}
        V_{1}(\epsilon)&\leq C\int_{[0,t]^{4}}|\epsilon I+\Sigma|^{-2}|\mu|drdadbdc \\
        &\leq C\int_{[0,t]^{3}}|\epsilon I+\Sigma|^{-2}|\mu|dadbdc.
    \end{align*}
    \begin{remark}
        For te rest of this article, the constant $C$ may differ from equation to equation or from line to line, but it does not affect the calculation thereafter. The subscripts of $C$ indicate what the constant depends on.
    \end{remark}
    Applying \Cref{Bounds}, there exists some constant $C>0$ such that 
    \begin{align*}
        |\epsilon I+\Sigma|&=(\epsilon+\Sigma_{1,1})(\epsilon+\Sigma_{2,2})-\Sigma_{1,2}^{2}\\
        &\geq C \left[\epsilon^{2}+\epsilon((a+b)^{2H}+(b+c)^{2H})+a^{2H}(c+b)^{2H}+c^{2H}(a+b)^{2H} \right]\\
        &\geq C[\epsilon^{2}+(a+b)^{H}(b+c)^{H}(\epsilon+a^{H}c^{H})]\\
        &\geq C (a+b)^{H}(b+c)^{H}(\epsilon+a^{H}c^{H}),
    \end{align*}
    where we use the Young's inequality in the second to last inequality. Note that 
    \begin{align*}
        |\mu|<\sqrt{\lambda\rho}=(a+b)^{H}(b+c)^{H}. 
    \end{align*}
    As such, when $\frac{1}{2}<H<1$,we have 
    \begin{align}
    \label{V_{1}}
    \begin{split}
        \limsup_{\epsilon \to 0}\frac{{V_{1}(\epsilon)}}{\epsilon^{\frac{2}{H}-4}}&\leq \limsup_{\epsilon \to 0} C \epsilon^{4-\frac{2}{H}} \int_{[0,t]^{3}}(a+b)^{-H}(b+c)^{-H}(\epsilon+a^{H}c^{H})^{-2}dadbdc \\
        &\leq \limsup_{\epsilon \to 0} C_{t} \epsilon^{4-\frac{2}{H}} \int_{[0,t]^{3}}b^{-H}c^{-H}(\epsilon+a^{H}c^{H})^{-2}dadbdc\\
        &\leq \limsup_{\epsilon \to 0} C_{t} \epsilon^{4-\frac{2}{H}} \epsilon^{\frac{1}{H}-2}\int_{[0,t\epsilon^{-\frac{1}{H}}]\times [0,t]}c^{-H}(1+u^{H}c^{H})^{-2}dudc\\
        &\leq \limsup_{\epsilon \to 0} C_{t,H} \epsilon^{2-\frac{1}{H}}\\
        &=0,
    \end{split}
    \end{align}
    where we change variables $a$  by $u\epsilon^{\frac{1}{H}}$. For the term $V_{2}(\epsilon)$, changing variables $(r,r^{'},s,s^{'})$ by ($r$, $r^{'}-r=a$, $s^{'}-r^{'}=b$, $s-s^{'}=c$), there exists some $C>0$ such that
    \begin{align*}
        V_{2}(\epsilon)\leq C\int_{[0,t]^{3}}|\epsilon I+\Sigma|^{-\frac{d}{2}-1}|\mu|dadbdc.
    \end{align*}
Note that 
\begin{align*}
    |\mu|&=\frac{1}{2}((a+b)^{2H}+(b+c)^{2H}-a^{2H}-c^{2H})\\
    &=Hb\int_{0}^{1}((a+bv)^{2H-1}+(c+bv)^{2H-1})dv\\
    &\leq 2Hb(\min(a,c))^{2H-1}\\
    &\leq 2Hb(a^{2H-1}+c^{2H-1}).
\end{align*}
According to \Cref{Bounds}, 
\begin{align*}
    |\epsilon I+\Sigma|&\geq \epsilon^{2}+\epsilon((a+b+c)^{2H}+b^{2H})+K_{2}b^{2H}(a^{2H}+c^{2H})\\
    &\geq C(\epsilon^{2}+\epsilon((a+b+c)^{2H}+b^{2H})+b^{2H}(a^{2H}+c^{2H}))\\
    &\geq  C(\epsilon^{2}+\epsilon((a+c)^{2H}+b^{2H})+b^{2H}(a+c)^{2H})\\
    &=C(\epsilon+(a+b)^{2H})(\epsilon+b^{2H}),
\end{align*}
where $C=\min(1,K_{2})$ in the second inequality.
Hence, we have
\begin{align*}
    V_{2}(\epsilon)&\leq C\int_{[0,t]^{3}}b(a^{2H-1}+c^{2H-1})(\epsilon+(a+b)^{2H})^{-2}(\epsilon+b^{2H})^{-2}dadbdc\\
    &= C\epsilon^{\frac{3}{2H}-3}\int_{[0,\epsilon^{-\frac{1}{2H}}t]^{3}}\frac{b}{(1+b^{2H})^{2}}\frac{a^{2H-1}+c^{2H-1}}{(1+(a+b)^{2H})^{2}}dadbdc\\
    &\leq C_{t,H}\epsilon^{\frac{3}{2H}-3},
\end{align*}
where we change variables $(a,b,c)$ by $(\epsilon^{\frac{1}{2H}}a,\epsilon^{\frac{1}{2H}}b,\epsilon^{\frac{1}{2H}}c)$ in the second inequality.
Hence, when $\frac{1}{2}<H<1$, 
\begin{align}
\label{V_{2}}
    \lim_{\epsilon\to 0}\epsilon^{4-\frac{2}{H}}V_{2}(\epsilon)=0.
\end{align}
Now we deal with term $V_{3}(\epsilon)$. Note that 
\begin{align*}
    V_{3}(\epsilon)&=\frac{2}{(2\pi)^{2}}\int_{[0,t]^{3}}1_{[0,t]}(a+b+c)(t-a-b-c)|\epsilon I+\Sigma|^{-2}|\mu|dadbdc\\
    &=\frac{2}{(2\pi)^{2}}\epsilon^{\frac{2}{H}-4}\int_{[0,\infty)^{3}}1_{[0,t]}(b+\epsilon^{\frac{1}{2H}}(a+c))(t-b-\epsilon^{\frac{1}{2H}}(a+c))\frac{\epsilon^{-\frac{1}{H}}\mu_{\epsilon}}{((1+a^{2H})(1+c^{2H})-\epsilon^{-2}\mu_{\epsilon}^{2})^{2}}dadbdc,
\end{align*}
where we change the variables $(a,b,c)$ by $(\epsilon^{\frac{1}{2H}}a,b,\epsilon^{\frac{1}{2H}}c)$ in the last equality, and
\begin{align*}
    \mu_{\epsilon}&=\frac{1}{2}|(b+\epsilon^{\frac{1}{2H}}a+\epsilon^{\frac{1}{2H}}c)^{2H}+b^{2H}-(\epsilon^{\frac{1}{2H}}a+b)^{2H}-(\epsilon^{\frac{1}{2H}}c+b)^{2H}|\\
    &=H(2H-1)\epsilon^{\frac{1}{H}}ac\int_{[0,1]^{2}}(b+\epsilon^{\frac{1}{2H}}av_1+\epsilon^{\frac{1}{2H}}cv_{2})^{2H-2}dv_{1}dv_{2}.
\end{align*}
The integration region has been changed from $[0,t]$ to $[0,\infty)$ since $\{(a,b,c); 0\leq a+b+c\leq t\}\subset [0,t]^{3}\subset [0,\infty)^{3}.$
Denote by
\begin{align*}
    \Phi_{\epsilon}:=\frac{\mu_{\epsilon}1_{[0,t]}(b+\epsilon^{\frac{1}{2H}}(a+c))(t-b-\epsilon^{\frac{1}{2H}}(a+c))\epsilon^{-\frac{1}{H}}}{[(1+a^{2H})(1+c^{2H})-\epsilon^{-2}\mu_{\epsilon}^{2}]^{2}}.
\end{align*}
Note that, when $\frac{1}{2}<H<1$,
\begin{align*}
    \lim_{\epsilon\to 0}\mu_{\epsilon}\epsilon^{-\frac{1}{H}}&=H(2H-1)acb^{2H-2}\\
    \lim_{\epsilon\to 0}\mu_{\epsilon}^{2}\epsilon^{-2}&=0.
\end{align*}
As such, 
\begin{align}
\label{integrand}
    \lim_{\epsilon\to 0}\Phi_{\epsilon}=\frac{1_{[0,t]}(b)H(2H-1)acb^{2H-2}(t-b)}{[(1+a^{2H})(1+c^{2H})]^{2}}.
\end{align}
Denote 
\begin{align*}
    \hat{V}_{3}&=\frac{2}{(2\pi)^{2}}\int_{[0,\infty)^{3}}\frac{1_{[0,t]}(b)H(2H-1)acb^{2H-2}(t-b)}{[(1+a^{2H})(1+c^{2H})]^{2}}dadbdc\\
    &=\frac{2}{(2\pi)^{2}}\int_{[0,\infty)^{2}}\frac{H(2H-1)ac}{[(1+a^{2H})(1+c^{2H})]^{2}}dadc\int_{[0,1]}t^{2H}b^{2H-2}(1-b)db,
\end{align*}
where we change $b$ by $tb$ in the second equality. By \Cref{betafunction},
\begin{align*}
    \hat{V}_{3}=\sigma_{1}^{2}.
\end{align*} 
If $\Phi_{\epsilon}$ is bounded by an integrable function in $\mathbb{R}_{+}^{3}$, by dominated convergence theorem, we have
\begin{align}
    \label{V_{3}}
    \lim_{\epsilon\to 0}\epsilon^{4-\frac{2}{H}}V_{3}(\epsilon)=\hat{V}_{3}.
\end{align}
Utilising \cref{integrand}, there exists a positive constant C depending only on $H$ and $t$ such that,
\begin{align*}
    \Phi_{\epsilon}\leq C\frac{acb^{2H-2}}{(1+a^{2H})^{2}(1+c^{2H})^{2}}.
\end{align*}
Obviously the expression on the right hand side is an integrable function in $\mathbb{R}_{+}^{3}$ given $\frac{1}{2}<H<1$.
Combining \cref{V_{1}},\cref{V_{2}} and \cref{V_{3}}, we have 
\begin{align*}
     \lim_{\epsilon\to 0}\mathbb{E}[\epsilon^{4-\frac{2}{H}}|\hat{\alpha}^{(1)}_{t,\epsilon}|^{2}]=\sigma_{1}^{2}.
\end{align*}
\end{proof}
\begin{lemma}
\label{lemm2}
When $d=3$, $\frac{1}{2}<H<\frac{2}{3}$, we have
\begin{align*}
    \lim_{\epsilon\to 0}\mathbb{E}[\epsilon^{5-\frac{2}{H}}|\hat{\alpha}^{(1)}_{t,\epsilon}|^{2}]=\sigma_{2}^{2},
\end{align*}
where $\sigma_{2}^{2}$ is defined in \Cref{thm2}.
\end{lemma}
\begin{proof}
    Similar to \Cref{lemm1}.
\end{proof}
\begin{lemma}
\label{lemma3}
    When $d=2$, $\frac{1}{2}<H<1$, we have 
    \begin{align*}
        \lim_{\epsilon\to 0}\mathbb{E}[\epsilon^{4-\frac{2}{H}}|I_{1}(f^{(1)}_{1,t,\epsilon})|^{2}]=\sigma_{1}^{2},
    \end{align*}
    where $I_{1}(f^{(1)}_{1,t,\epsilon})$ is the first chaos of $\alpha_{t,\epsilon}^{(1)}$ and $f^{(1)}_{1,t,\epsilon}$ is defined in \Cref{explict form of f}.
    \begin{proof}
        Since $I_{1}(\cdot)$ is an isometry from $\mathcal{H}^{2}$ to $L^{2}(\Omega)$, we have
        \begin{align*}
            \mathbb{E}[|I_{1}(f_{1,t,\epsilon}^{(1)})|^{2}]&=||f_{1,t,\epsilon}^{(1)}||_{\mathcal{H}^{2}}^{2}\\
            &=\sum_{i=1}^{3}\tilde{V}_{i}(\epsilon),
        \end{align*}
        where
        \begin{align*}
            \tilde{V}_{i}(\epsilon)=\frac{2}{(2\pi)^{2}}\int_{D_{i}}\frac{\langle1_{[r,s]},1_{[r^{'},s^{'}]}\rangle_{\mathcal{H}}}{((s-r)^{2H}+\epsilon)^{2}((s^{'}-r^{'})^{2H}+\epsilon)^{2}}drdsdr^{'}ds^{'},
        \end{align*}
        and $D_{i}, 1\leq i\leq 3$ are regions defined in \Cref{Bounds}. $\tilde{V}_{3}(\epsilon)$ will be dealt with first as we will see later it is the dominant term.  
        Utilising the fact that 
        \begin{align*}
            \langle1_{[r_{1},s_{1}]},1_{[r_{2},s_{2}]}\rangle_{\mathcal{H}}=\mathbb{E}[(B^{H,1}_{s_{1}}-B^{H,1}_{r_{1}})(B^{H,1}_{s_{2}}-B^{H,1}_{r_{2}})],
        \end{align*}
        we have
        \begin{align*}
            \tilde{V}_{3}(\epsilon)=\frac{2}{(2\pi)^{2}}\int_{D_{3}}\frac{\mathbb{E}[(B^{H,1}_{s}-B^{H,1}_{r})(B^{H,1}_{s^{'}}-B^{H,1}_{r^{'}})]}{((s-r)^{2H}+\epsilon)^{2}((s^{'}-r^{'})^{2H}+\epsilon)^{2}}drdsdr^{
            '}ds^{'}.
        \end{align*}
        According to \Cref{Bounds}, 
        \begin{align*}
            \tilde{V}_{3}(\epsilon)=\frac{2}{(2\pi)^{2}}\int_{[0,t]^{3}}\frac{1_{[0,t]}(a+b+c)(t-a-b-c)\frac{1}{2}((a+b+c)^{2H}+b^{2H}-(a+b)^{2H}-(c+b)^{2H})}{(\epsilon+a^{2H})^{2}(\epsilon+c^{2H})^{2}}dadbdc.
        \end{align*}
        Change variables $(a,b,c)$ by $(\epsilon^{\frac{1}{2H}}a,b,\epsilon^{\frac{1}{2H}}c)$, we have
        \begin{align*}
              \tilde{V}_{3}(\epsilon)=\frac{2}{(2\pi)^{2}}\epsilon^{\frac{2}{H}-4}\int_{[0,\infty]^{3}}1_{[0,t]}(b+\epsilon^{\frac{1}{2H}}(a+c))(t-b-\epsilon^{\frac{1}{2H}}(a+c))\frac{\epsilon^{-\frac{1}{H}}\mu_{\epsilon}}{(1+a^{2H})^{2}(1+c^{2H})^{2}}dadbdc,
        \end{align*}
        where 
        \begin{align*}
              \mu_{\epsilon}&=\frac{1}{2}|(b+\epsilon^{\frac{1}{2H}}a+\epsilon^{\frac{1}{2H}}c)^{2H}+b^{2H}-(\epsilon^{\frac{1}{2H}}a+b)^{2H}-(\epsilon^{\frac{1}{2H}}c+b)^{2H}|\\
    &=H(2H-1)\epsilon^{\frac{1}{H}}ac\int_{[0,1]^{2}}(b+\epsilon^{\frac{1}{2H}}av_1+\epsilon^{\frac{1}{2H}}cv_{2})^{2H-2}dv_{1}dv_{2}.
        \end{align*}
        Clearly,
        \begin{align*}
            \lim_{\epsilon\to 0}\epsilon^{-\frac{1}{H}}\mu_{\epsilon}&=H(2H-1)acb^{2H-2},\\
             \lim_{\epsilon\to 0}\frac{1_{[0,t]}(b+\epsilon^{\frac{1}{2H}}(a+c))(t-b-\epsilon^{\frac{1}{2H}}(a+c))\epsilon^{-\frac{1}{H}}\mu_{\epsilon}}{(1+a^{2H})^{2}(1+c^{2H})^{2}}&=1_{[0,t]}(b)(t-b)\frac{H(2H-1)acb^{2H-2}}{(1+a^{2H})^{2}(1+c^{2H})^{2}}.
        \end{align*}
        Therefore, there exists a positive constant $C$ such that
        \begin{align*}
            1_{[0,t]}(b+\epsilon^{\frac{1}{2H}}(a+c))(t-b-\epsilon^{\frac{1}{2H}}(a+c))\frac{\epsilon^{-\frac{1}{H}}\mu_{\epsilon}}{(1+a^{2H})^{2}(1+c^{2H})^{2}}\leq C\frac{H(2H-1)acb^{2H-2}}{(1+a^{2H})^{2}(1+c^{2H})^{2}},
        \end{align*}
        in which the expression on the right hand side is an integrable function in $\mathbb{R}_{+}^{3}$.
        Hence, by the dominated convergence theorem,
        \begin{align}
        \begin{split}
            \label{tV_{3}}       \lim_{\epsilon\to0}\epsilon^{4-\frac{2}{H}} \tilde{V}_{3}(\epsilon)&=\frac{2}{(2\pi)^{2}}\int_{[0,\infty]^{3}}1_{[0,t]}(b)(t-b)\frac{H(2H-1)acb^{2H-2}}{(1+a^{2H})^{2}(1+c^{2H})^{2}}dadbdc \\
            &=\frac{2H(2H-1)t^{2H}}{(2\pi)^{2}}\int_{[0,\infty]{2}}\frac{ac}{(1+a^{2H})^{2}(1+c^{2H})^{2}}dadc\int_{[0,1]}(1-b)b^{2H-2}db\\
            &=\sigma_{1}^{2},    
        \end{split}
            \end{align}
        where we change variable $b$ by $tb$ in the second equality, and we utilize \Cref{betafunction} in the last equality. According to \Cref{Chaos expansion for L2} and \Cref{explict form of f}, we have  
        \begin{align*}
         \sum_{q=1}^{\infty}\mathbb{E}[|I_{2q-1}(f^{(1)}_{2q-1,t,\epsilon})|^{2}]=   \mathbb{E}[|\hat{\alpha}_{t,\epsilon}^{(1)}|^{2}]=
            V_{1}(\epsilon)+V_{2}(\epsilon)+V_{3}(\epsilon).
        \end{align*}
        This implies 
        \begin{align*}
            \tilde{V}_{1}(\epsilon)+\tilde{V}_{2}(\epsilon)+\tilde{V}_{3}(\epsilon)=\mathbb{E}[|I_{1}(f^{1}_{1,t,\epsilon})|^{2}]\leq V_{1}(\epsilon)+V_{2}(\epsilon)+V_{3}(\epsilon).
        \end{align*}
        As such, we have
        \begin{align*}
            \lim_{\epsilon\to 0}\epsilon^{4-\frac{2}{H}}(\tilde{V_{1}}(\epsilon)+\tilde{V_{2}}(\epsilon)+\tilde{V}_{3}(\epsilon))&\leq  \lim_{\epsilon\to 0}\epsilon^{4-\frac{2}{H}}(V_{1}(\epsilon)+V_{2}(\epsilon)+V_{3}(\epsilon))\\
            \lim_{\epsilon\to 0}\epsilon^{4-\frac{2}{H}}(\tilde{V_{1}}(\epsilon)+\tilde{V_{2}}(\epsilon))&\leq  \lim_{\epsilon\to 0}\epsilon^{4-\frac{2}{H}}(V_{1}(\epsilon)+V_{2}(\epsilon))=0,
        \end{align*}
        where $\tilde{V_{3}}(\epsilon)$ and $V_{3}(\epsilon)$ are cancelled due to the fact that they all converge to $\sigma_{1}^{2}$ by \cref{V_{3}} and \cref{tV_{3}}, and the last equality holds due to \cref{V_{1}} and \cref{V_{2}}.
        Hence 
         \begin{align*}
        \lim_{\epsilon\to 0}\mathbb{E}[\epsilon^{4-\frac{2}{H}}|I_{1}(f^{(1)}_{1,t,\epsilon})|^{2}]=\sigma_{1}^{2},
    \end{align*}
    as required.
    \end{proof}
\end{lemma}
\begin{lemma}
\label{lemma4}
    When $d=2$, $\frac{1}{2}<H<\frac{2}{3}$, we have
    \begin{align*}
        \lim_{\epsilon\to 0}\mathbb{E}[\epsilon^{5-\frac{2}{H}}|I_{1}(f^{(1)}_{1,t,\epsilon})|^{2}]=\sigma_{2}^{2}.
    \end{align*}
\begin{proof}
    Similar to \Cref{lemma3}.
\end{proof}
\end{lemma}
Now we prove \Cref{thm1} and \Cref{thm2}. 
\begin{proof}[Proof of \Cref{thm1}]
Choosing $d=2$ and $\frac{1}{2}<H<1$ and applying \Cref{lemm1} and \Cref{lemma3}, we have
\begin{align*}
    \lim_{\epsilon\to 0}\mathbb{E}[\epsilon^{4-\frac{2}{H}}|\hat{\alpha}^{(1)}_{t,\epsilon}|^{2}]=\lim_{\epsilon\to 0}\mathbb{E}[\epsilon^{4-\frac{2}{H}}|I_{1}(f^{(1)}_{1,t,\epsilon})|^{2}],
\end{align*}
and according to \Cref{explict form of f}, this means that the term
\begin{align*}
    \epsilon^{2-\frac{1}{H}}\sum_{q=2}^{\infty}I_{2q-1}(f^{(1)}_{2q-1,t,\epsilon})
\end{align*}
converges to 0 in $L^{2}(\Omega)$. Since $\epsilon^{2-\frac{1}{H}}I_{1}(f_{1,t,\epsilon}^{(1)})$ is Gaussian and its variance converges to $\sigma_{1}^{2}$, then \Cref{thm1} follows.
\end{proof}
\begin{proof}[Proof of \Cref{thm2}] Choosing $d=3$ and $\frac{1}{2}<H<\frac{2}{3}$ and applying \Cref{lemm2} and \Cref{lemma4}, we have 
\begin{align*}
    \lim_{\epsilon\to 0}\mathbb{E}[\epsilon^{5-\frac{2}{H}}|\hat{\alpha}^{(1)}_{t,\epsilon}|^{2}]=\lim_{\epsilon\to 0}\mathbb{E}[\epsilon^{5-\frac{2}{H}}|I_{1}(f^{(1)}_{1,t,\epsilon})|^{2}],
\end{align*}
and according to \Cref{explict form of f}, this means
\begin{align*}
    \epsilon^{\frac{5}{2}-\frac{1}{H}}\sum_{q=2}^{\infty}I_{2q-1}(f^{(1)}_{2q-1,t,\epsilon})
\end{align*}
converges to 0 in $L^{2}(\Omega)$. Since $\epsilon^{\frac{5}{2}-\frac{1}{H}}I_{1}(f_{1,t,\epsilon}^{(1)})$ is Gaussian and its variance converges to $\sigma_{2}^{2}$, then \Cref{thm2} follows.
\end{proof}
\subsection*{Acknowledgements}
We would like to thank the anonymous referees for valuable comments.

\subsection*{Declaration of interest}
Declarations of interest: none.

\subsection*{Funding}
Binghao Wu acknowledges funding from an Australian Government Research Training Program (RTP) Scholarship. Kaustav Das has been supported by the Australian Research Council (Grant DP220103106). Qian Yu is supported by the National Natural Science Foundation of China (12201294).

\appendix
\section{Malliavin calculus preliminaries} \label{malliavin appendix}
     In this article, the Hilbert spaces $\mathcal{H}$ discussed are separable with an inner product $\langle \cdot, \cdot \rangle_{\mathcal{H}}$. We denote the norm of an element $h\in\mathcal{H}$ by $||\cdot||_{\mathcal{H}}$.
\begin{def}
\label{isonormal Gaussian processes}
    We say a stochastic process $W=\{W(h);h\in \mathcal{H}\}$ defined on a complete probability space $(\Omega, \mathcal{F}, \mathbb{P})$ is an isonormal Gaussian process if $W$ is a centered Gaussian family of random variables with $\mathbb{E}[W(h)W(g)]=\langle h,g\rangle_{\mathcal{H}}$ for all $h,g \in \mathcal{H}$.
\end{def}
\noindent
Let $H_q$ denote the $q$-th Hermite polynomial, defined as
\begin{align*}
    H_{q}(x)=(-1)^{q}e^{\frac{x^{2}}{2}}\frac{d^{q}}{dx^{q}}e^{-\frac{x^{2}}{2}}, \quad q\geq1,
\end{align*}
and $H_{0}(x)=1$. 
\begin{lemma}
    Let $X$ and $Y$ be two jointly Gaussian random variables with mean zero and variance $1$, then for $n,m\geq 1$, we have
    \begin{align*}
        \mathbb{E}[H_{n}(X)H_{m}(Y)]=
        \begin{cases}
            0 &\text{if $m\neq n$ }\\
            n!\mathbb{E}[XY]^{n}&\text{ if } m=n.
        \end{cases}
    \end{align*}
\end{lemma}
\begin{dfn}
    A topological vector space $A$ is said to be a total subset of $B$ if 
    \begin{align*}
        \overline{\text{Span}(A)}=B.
    \end{align*}
\end{dfn}
\noindent
Let $\mathcal{G}$ be the $\sigma$-algebra generated by the collection of random variables $\{W(h);h\in\mathcal{H}\}$.
\begin{lemma}
    The random variables $\{e^{W(h)};h\in \mathcal{H}\}$ form a total subset of $L^{2}(\Omega,\mathcal{G},\mathbb{P})$.
\end{lemma}
\noindent
\begin{dfn}
Denote by $\mathbb{H}_{n}$ the closed subspace of $L^{2}(\Omega,\mathcal{G},\mathbb{P})$ generated by the random variables $\{H_{n}(W(h)); h\in\mathcal{H},||h||_{\mathcal{H}}=1\}$ with $n\geq 0$, that is
\begin{align*}
    \mathbb{H}_{n}=\overline{\text{Span}(\{H_{n}(W(h));h\in \mathcal{H},||h||_{\mathcal{H}}=1\})}.
\end{align*}
\end{dfn}
\begin{thm}
\label{Chaos expansion for L2}
    The space $L^{2}(\Omega,\mathcal{G},\mathbb{P})$ can be decomposed as an infinite direct sum of subspaces $\mathbb{H}_{n}$:
    \begin{align*}
        L^{2}(\Omega,\mathcal{G},\mathbb{P})=\bigoplus_{n=0}^{\infty}\mathbb{H}_{n}.
    \end{align*}
\end{thm}
\noindent
Let $C_{p}^{\infty}(\R^{n})$ be the set of infinitely differentiable functions $f:\R^{n}\to \R$ such that all of its partial derivatives have at most polynomial growth. Denote by $S$ the class of smooth random variables that has a form 
\begin{align*}
    F=f(W(h_{1}),\dots,W(h_{n}))
\end{align*}
with $f\in C_{p}^{\infty}(\R^{n})$ and $h_{1},\dots,h_{n}\in \mathcal{H}$. We will use notation $\partial_{i}f$ to denote $\frac{\partial f}{\partial x_{i}}$.
\begin{dfn}
    The derivative of a smooth random variable $F\in S$ is an $\mathcal{H}$ valued random variable:
    \begin{align*}
        DF=\sum_{i=1}^{n}\partial_{i}f(W(h_{1}),\dots,W(h_{n}))h_{i}.
    \end{align*}
\end{dfn}
\begin{prop}
    The operator $D$ is closable from $L^{p}(\Omega)$ to $L^{p}(\Omega;\mathcal{H})$ for all $p\geq 1$.
\end{prop}
\noindent
We denote the domain of the operator $D$ in $L^{p}(\Omega)$ by $\mathbb{D}^{1,p}$ meaning it is the closure of $S$ with respect to the semi-norm defined as 
\begin{align*}
    ||F||_{1,p}=\left|\mathbb{E}[|F|^{p}]+\mathbb{E}[||DF||_{\mathcal{H}}^{p}]\right|^{\frac{1}{p}}. 
\end{align*}
The $k$-th iteration of $D$ for a smooth random variable $F$ is denoted as $D^{k}F$, which is an $\mathcal{H}^{\otimes k}$ valued random variable. As such $\mathbb{D}^{k,p}$ is the closure of $S$ with respective to the semi-norm 
\begin{align*}
    ||F||_{k,p}=\left|\mathbb{E}[|F|^{p}]+\mathbb{E}[||D^{k}F||_{\mathcal{H}^{\otimes k}}^{p}]\right|^{\frac{1}{p}}.
\end{align*}
Denote
\begin{align*}
\mathbb{D}^{\infty}=\bigcap_{k=1,p=1}^{\infty}\mathbb{D}^{k,p}.
\end{align*}
For more details in Malliavin Calculus, please refer to \citep{nualart2006malliavin}.
 
The $k$th DSLT of fractional Brownian motion is defined as follows:
\begin{align*}
    \lim_{\epsilon\to0}\hat{\alpha}^{(k)}_{t,\epsilon}=\lim_{\epsilon\to 0}\int_{D}\delta^{(k)}_{\epsilon}(B_{s}^{H}-B_{r}^{H})drds,
\end{align*}
where $D=\{(r,s)|0<r<s<t\}$, $\{B_{t}^{H}=(B^{H,1}_{t},\dots,B^{H,d}_{t})\}_{t\geq0}$ is a $d$-dimensional fractional Brownian motion
with $k=(k_{1},\dots,k_{d})$ and $|k|=\sum_{j=1}^{d}k_{j}$. Consider the space of indicator functions \begin{align*}
    \mathcal{L}=\{1_{[a,b]}; a,b\in\R, a\leq b\}.
\end{align*} 
Let $\mathcal{H}$ be the Hilbert space obtained by completing $\mathcal{L}$ with respect to the inner product
\begin{align*}
\langle1_{[a,b]},1_{[c,d]}\rangle_{\mathcal{H}}=\mathbb{E}[(B^{H,1}_{b}-B_{a}^{H,1})(B^{H,1}_{d}-B_{c}^{H,1})].    
\end{align*}
For all $f=(f_{1},\dots,f_{d})\in \mathcal{H}^{d}$, we define 
\begin{align*}
    B^{H}(f)=\sum_{j=1}^{d}B^{H,j}(f_{j}).
\end{align*}
Each $B^{H,j}(\cdot)$ is the isonormal Gaussian process with the associated Hilbert space $\mathcal{H}$.
As such $B^{H}(\cdot)$ is an isometry from $\mathcal{H}^{d}$ to the Gaussian subspace of $L^{2}({\Omega})$ generated by the $d$-dimensional fractional Brownian motion. The $q$-th Wiener chaos of $L^{2}(\Omega)$, denoted as $\mathbb{H}_{q}$, is a closed subspace of $L^{2}(\Omega)$ generated by the random variables 
\begin{align*}
    \left\{\prod_{j=1}^{d}H_{q_{j}}(B^{H,j}(f_{j}));\sum_{j=1}^{d}q_{j}=q, f_{j}\in \mathcal{H},||f_{j}||_{\mathcal{H}}=1\right\},
\end{align*}
where $H_{q}$ is the $q$th Hermite polynomial. For every $q\in\mathbb{N}$, we denote by $(\mathcal{H}^{d})^{\otimes q}$ the $q$-th tensor product of $\mathcal{H}^{d}$.


For $f^{1},\dots,f^{q}\in \mathcal{H}^{d}$ of the form $f^{i}=(f_{1}^{i},\dots,f_{d}^{i})$ with $ 1\leq i\leq q$, $f^{1}\otimes\cdots\otimes f^{q}$ can be defined as a multi-dimensional array:
\begin{align}
\label{eqn:tensor1}
    f^{1}\otimes\cdots\otimes f^{q}=(f{_{i_{1}}}^{1}\otimes\cdots\otimes f_{i_{q}}^{q})_{i_{1},\dots,i_{q}=1,\dots,d}.
\end{align}
The tensor product \cref{eqn:tensor1} is isomorphic to following form of the tensor product:
\begin{align*}
	f^{1}\otimes\cdots\otimes f^{q} = \sum_{i_1, \dots, i_q = 1}^d F_{i_1}^1 \otimes F_{i_2}^2 \otimes \dots \otimes F_{i_q}^q
\end{align*}
where $F_i^j = (0, \dots, f_i^j, \dots, 0)$ is a tuple of size $d$, which is equal to $f_i^j$ in the $i$-th position, and zero elsewhere.

In the special case that $f^1 = f^2 = \dots = f^q$, we then have
\begin{align}
\label{qth tensor product}
	f^{\otimes q} =  \sum_{i_1, \dots, i_q = 1}^d F_{i_1} \otimes F_{i_2} \otimes \dots \otimes F_{i_q}
\end{align}
where $F_i = (0, \dots, f_i, 0, \dots, 0)$ is a tuple of size $d$, which is equal to $f_i$ in the $i$-th position, and zero elsewhere.
We will prefer to use this form of the tensor product, as handling sums is more computationally convenient than multi-dimensional arrays.
Denote the symmetrization of $(\mathcal{H}^{d})^{\otimes q}$ by $(\mathcal{H}^{d})^{\odot q}$. Let $f\in \mathcal{H}^{d}$ be of the form $f=(f_{1},\dots,f_{d})$ with $||f_{j}||_{\mathcal{H}}=1$. Such $f^{\otimes q}$ belongs to $(\mathcal{H}^{d})^{\odot q}$, and we can define a mapping $I_{q}:(\mathcal{H}^{d})^{\odot q}\to \mathbb{H}_{q} $ as follows
\begin{align*}
    I_{q}(f^{\otimes q})=\sum_{i_{1},\dots,i_{q}=1}^{d}\sqrt{q_{1}(i_{1},\dots,i_{q})!\cdots q_{d}(i_{1},\dots,i_{q})!}\prod_{j=1}^{d}H_{q_{j}(i_{1},\dots,i_{q})}(B^{H, j}(f_{j})), j=1,\dots,d
\end{align*}
where $q_{j}(i_{1},\dots,i_{q})$ denotes the number of indices in $(i_{1},\dots,i_{q})$ equal to $j$. This mapping is a linear isometry between  $(\mathcal{H}^{d})^{\odot q}$ and $\mathbb{H}_{q}$. Thus, by \Cref{Chaos expansion for L2}, any square integrable random variable $F$ which is measurable with respect to the $\sigma$-algebra generated by the fractional Brownian motion will have a chaos expansion of the type
\begin{align*}
    F=\mathbb{E}[F]+\sum_{q=1}^{\infty}I_{q}(g_{q})
\end{align*}
 for some $g_{q}\in (\mathcal{H}^{d})^{\odot q}$.
\begin{lemma}
\label{explict form of f}
    Let $k=(|k|,0,\dots,0)\in \mathbb{N}^{d}$ with $d\in \mathbb{N}$ and $|k|\geq 1$ being odd. Then $\hat{\alpha}_{\epsilon,t}^{(k)}$ defined in $\cref{DSLT of fbm}$ possesses a Wiener chaos expansion,
    \begin{align*}
        \hat{\alpha}^{(k)}_{t,\epsilon}=\sum_{q=1}^{\infty}I_{2q-1}(f^{(k)}_{2q-1,t,\epsilon}),
    \end{align*}
    where 
\begin{align*}
f^{(k)}_{q,t,\epsilon}&=\frac{(-1)^{\frac{|k|+q}{2}}}{(2\pi)^{\frac{d}{2}}}\sum_{i_{1}=1,\dots,i_{q}=1}^{d}
     \int_{D} h_{i_{1}}\otimes\cdots\otimes h_{i_{q}}\\
    &\times \frac{(|k|+q_{1}(i_{1},\dots,i_{q})-1)!!\times\cdots\times (q_{d}(i_{1},\dots,i_{q})-1)!!}{((s-r)^{2H}+\epsilon)^{\frac{|k|+q+d}{2}}}drds
\end{align*}
with $h_{i}=(0,\dots,1_{[r,s]},\dots,0)\in \mathcal{H}^{d}$ which is zero everywhere except at its $i$-th entry.
\end{lemma}
\begin{remark}
    The proof here adopts similar techniques in Lemma 7 in \citep{hu2005renormalized}, Appendix $A$ in \citep{das2022existence} and Lemma 2.2 in \citep{yu2024limit}. Since we are interested in the limit theorem of $\hat{\alpha}_{t,\epsilon}^{(1)}$ in which $|k|=1$.  We  assume $k=(|k|,0,\dots,0)$ and $|k|\geq 1$ is odd.
\end{remark}
\begin{proof}
Note that we can rewrite the derivative of self-intersection local time of fractional Brownian motion as  
\begin{align*}
    \hat{\alpha}^{(k)}_{t,\epsilon}&=\frac{i^{|k|}}{(2\pi)^{d}}\int_{D}\int_{\R^{d}}\prod_{j=1}^{d}p_{j}^{k_{j}}e^{ip_{j}(B_{s}^{j,H}-B_{r}^{j,H})}e^{-\epsilon\frac{|p|^{2}}{2}}dpdrds\\
    &=\frac{i^{|k|}}{(2\pi)^{d}}\int_{D}\int_{\R^{d}}\prod_{j=1}^{d}p_{j}^{k_{j}}e^{B^{H}(h)}e^{-\epsilon\frac{|p|^{2}}{2}}dpdrds,
\end{align*}
where $h=(ip_{1}1_{[r,s]},\dots,ip_{d}1_{[r,s]})$. $e^{B^{H}(h)}$ is obviously in $\mathbb{D}^{\infty}$, and its $q$-th Malliavin derivative is 
\begin{align*}
    D^{q}e^{B(h)}=e^{B(h)}h^{\otimes q},
\end{align*}
where $h^{\otimes q}$ is defined through \cref{qth tensor product}. We can untangle $h^{\otimes q}$ as follows:
\begin{align}
\label{h^{q}}
\begin{split}
    h^{\otimes q}&=\sum_{i_{1},\dots,i_{q}=1}^{d}(0,..,ip_{i_{1}}1_{[r,s]},\dots0)\otimes\cdots\otimes(0,\dots,ip_{i_{q}}1_{[r,s]},\dots,0)\\
    &=\sum_{i_{1},\dots,i_{q}=1}^{d}i^{q}p_{1}^{q_{1}(i_{1},\dots,i_{d})}\cdots p_{d}^{q_{d}(i_{1},\dots,i_{d})}h_{i_{1}}\otimes\cdots\otimes h_{i_{q}},
\end{split}
\end{align}
where $h_{i}=(0,\dots,1_{[r,s]},\dots,0)\in \mathcal{H}^{d}$ only has non zero at its $i$-th entry.
Thus, by Stroock's formula, the chaos expansion for $\alpha^{(k)}_{t,\epsilon}$ is 
\begin{align*}
\alpha^{(k)}_{t,\epsilon}=\mathbb{E}[\alpha_{t,\epsilon}^{(k)}]+\sum_{q=1}^{\infty}I_{q}(f^{(k)}_{q,t,\epsilon}),
\end{align*}
where
\begin{align*}
f^{(k)}_{q,t,\epsilon}&=\frac{i^{|k|}}{(2\pi)^{d}}\int_{D}\int_{\R^{d}}\prod_{j=1}^{d}p_{j}^{k_{j}}\frac{1}{q!}\mathbb{E}[D^{q}e^{B^{H}(h)}]e^{-\epsilon\frac{|p|^{2}}{2}}dpdrds\\
&=\frac{i^{|k|}}{(q!)(2\pi)^{d}}\int_{D}\int_{\R^{d}}\prod_{j=1}^{d}p_{j}^{k_{j}}\mathbb{E}[e^{B^{H}(h)}]h^{\otimes q}e^{-\epsilon\frac{|p|^{2}}{2}}dpdrds\\
&=\frac{i^{|k|}}{(q!)(2\pi)^{d}}\int_{D}\int_{\R^{d}}\prod_{j=1}^{d}p_{j}^{k_{j}}e^{-\frac{1}{2}p^{2}_{j}((s-r)^{2H}+\epsilon)}h^{\otimes q}dpdrds.
\end{align*}
 Then by \cref{h^{q}}, we have 
\begin{align*}
    f^{(k)}_{q,t,\epsilon}&=\frac{i^{|k|}}{(q!)(2\pi)^{d}}\int_{D}\int_{\R^{d}}p_{1}^{|k|}e^{-\frac{1}{2}\sum_{j=1}^{d}p_{j}^{2}((s-r)^{2H}+\epsilon)}h^{\otimes q}dpdrds\\
    &=\frac{i^{|k|+q}}{(q!)(2\pi)^{d}}\sum_{i_{1}=1,\dotsi_{q}=1}^{d}\int_{D}\int_{\R^{d}}p_{1}^{|k|+q_{1}(i_{1},\dots,i_{q})}\cdots p_{d}^{q_{d}(i_{1},\dots,i_{q})}e^{-\frac{1}{2}\sum_{j=1}^{d}p_{j}^{2}((s-r)^{2H}+\epsilon)}\\ &\times h_{i_{1}}\otimes\cdots\otimes h_{i_{q}}dpdrds.
\end{align*}
Note that 
\begin{align*}
    \int_{\R^{d}}p_{1}^{|k|+q_{1}}\cdots p_{d}^{q_{d}}e^{-\frac{1}{2}\sum_{j=1}^{d}p_{j}^{2}((s-r)^{2H}+\epsilon)}dp=\frac{(|k|+q_{1}-1)!!\times\cdots\times (q_{d}-1)!!}{((s-r)^{2H}+\epsilon)^{\frac{|k|+q+d}{2}}}(2\pi)^{\frac{1}{2}d}
\end{align*}
when $|k|+q_{1},\dots,q_{d}$ are all even, otherwise it is equal to 0. Consequently, $f_{q,t,\epsilon}$ is 0 when $q$ is even. Thus, when $q$ is odd, we have 
\begin{align*}
f^{(k)}_{q,t,\epsilon}&=\frac{(-1)^{\frac{|k|+q}{2}}}{(2\pi)^{\frac{d}{2}}}\sum_{i_{1}=1,\dots,i_{q}=1}^{d}
     \int_{D} h_{i_{1}}\otimes\cdots\otimes h_{i_{q}}\\
    &\times \frac{(|k|+q_{1}(i_{1},\dots,i_{q})-1)!!\times\cdots\times (q_{d}(i_{1},\dots,i_{q})-1)!!}{((s-r)^{2H}+\epsilon)^{\frac{|k|+q+d}{2}}}drds.
\end{align*}
Hence when $|k|$ is odd, we have 
\begin{align}
\hat{\alpha}^{(k)}_{t,\epsilon}&=\mathbb{E}[\hat{\alpha}_{t,\epsilon}^{(k)}]+\sum_{q=1}^{\infty}I_{2q-1}(f^{(k)}_{2q-1,t,\epsilon})\\
&=\sum_{q=1}^{\infty}I_{2q-1}(f^{(k)}_{2q-1,t,\epsilon}),
\end{align}
as one can easily verify that $\mathbb{E}[\hat{\alpha}_{t,\epsilon}^{(k)}]=0$
\end{proof}
\section{Technical lemmas}
\label{appen:misc}
In this section we collect some of the technical estimates and facts which were used in the proofs of the theorems. Many of these facts can be found elsewhere, but we include proofs of most of them for the benefit of the reader.

The standard Gamma function is defined as follows:
\begin{align*}
	\Gamma(x)=\int_{0}^{\infty}t^{x-1}e^{-t}dt.
\end{align*}

This function is well defined except for negative integers, and satisfies $x\Gamma(x) = \Gamma(x+1)$. In this article we will only need to utilise the Gamma function with positive arguments. We require the following fact. 
\begin{lemma}[{\citep[Lemma ~A.3]{das2025exponential}}]
\label{lmag}
For any integer $n$ and $k\in (0,1)$,
\begin{align*}
\Gamma(kn)\leq((n-1)!)^{k},  \\
\Gamma(kn+1)\geq k^{n}(n!)^{k}.
\end{align*}
\end{lemma}
\begin{lemma}[{\citep[Lemma ~4.5]{hu2015stochastic}}]
\label{Beta function}
Let $\alpha\in (-1+\epsilon,1)^{m}$ with $\epsilon>0$ and set $|\alpha|=\sum_{i=1}^{m}\alpha_{i}$. $T_{m}(t)=\{(r_{1},r_{2},\dots,r_{m})\in \R^{m}:0<r_{1}<\dots<r_{m}<t\}$. Then there is a constant $c$ such that 
$$J_{m}(t,\alpha)=\int_{T_{m}(t)}\prod_{i=1}^{m}(r_{i}-r_{i-1})^{\alpha_{i}}dr\leq \frac{c^{m}t^{|\alpha|+m}}{\Gamma(|\alpha|+m+1)},$$
where by convention, $r_{0}=0$.
\end{lemma}    
\begin{lemma}[{\citep[Lemma A.1]{nualart2007intersection}}]
\label{Conditional Variance inequality}
    Suppose that $\mathcal{G}_{1}\subset\mathcal{G}_{2}$ are two $\sigma$-algebras in $\mathcal{F}$. Then for any square integrable random variable $F$ we have
    \begin{align*}
        \text{Var}(F|\mathcal{G}_{1})\geq \text{Var}(F|\mathcal{G}_{2})
    \end{align*}
    holds almost surely.
\end{lemma}
\begin{lemma}[{Appendix B in \citep{jung2014tanaka}}]
\label{Bounds}
    Let 
    \begin{align*}
        \lambda=|s-r|^{2H}, \rho=|s^{'}-r^{'}|^{2H},
    \end{align*}
    and
    \begin{align*}
        \mu=\frac{1}{2}\left(|s^{'}-r|^{2H}+|s-r^{'}|^{2H}-|s^{'}-s|^{2H}-|r-r^{'}|^{2H}\right).     
    \end{align*}
    \begin{itemize}
        \item Case(i) Suppose that $D_{1}=\{(r,r^{'},s,s^{'})\in [0,t]^{4}|r<r^{'}<s<s^{'}\}$, let $r^{'}-r=a$, $s-r^{'}=b$, $s^{'}-s=c$. Then, there exists a positive constant $K_{1}$ such that 
        \begin{align*}
            \lambda\rho-\mu^{2}\geq K_{1}((a+b)^{2H}c^{2H}+a^{2H}(b+c)^{2H})
        \end{align*}
        and
        \begin{align*}
            \mu=\frac{1}{2}((a+b+c)^{2H}+b^{2H}-a^{2H}-c^{2H}).
        \end{align*}
        \item Case(ii) Suppose that $D_{2}=\{(r,r^{'},s,s^{'})\in [0,t]^{4}|r<r^{'}<s^{'}<s\}$, let $r^{'}-r=a$, $s^{'}-r^{'}=b$, $s-s^{'}=c$. Then, there exists a positive constant $K_{2}$ such that \begin{align*}
            \lambda\rho-\mu^{2}\geq K_{2}b^{2H}(c^{2H}+a^{2H})
        \end{align*}
        and
        \begin{align*}
            \mu=\frac{1}{2}((a+b)^{2H}+(b+c)^{2H}-a^{2H}-c^{2H}).
        \end{align*}
        \item Case(iii) Suppose that $D_{2}=\{(r,r^{'},s,s^{'})\in [0,t]^{4}|r<s<r^{'}<s^{'}\}$, let $s-r=a$, $r^{'}-s=b$, $s^{'}-r^{'}=c$. Then, there exists a positive constant $K_{3}$ such that \begin{align*}
            \lambda\rho-\mu^{2}\geq K_{3}c^{2H}a^{2H}
        \end{align*}
        and \
        \begin{align*}
            \mu=\frac{1}{2}((a+b+c)^{2H}+b^{2H}-(a+b)^{2H}-(c+b)^{2H}).
        \end{align*}
    \end{itemize}
\end{lemma}
\begin{lemma}[{Lemma 5.5 in \citep{jaramillo2017asymptotic}}]
\label{betafunction}
    Let $c,\beta,\alpha$ and $\gamma$ be real numbers such that $c,\beta>0$, $\alpha>-1$ and $1+\alpha+\gamma\beta<0$. Then we have
    \begin{align*}
        \int_{0}^{\infty}a^{\alpha}(c+a^{\beta})^{\gamma}da=\beta^{-1}c^{\frac{1+\alpha+\gamma\beta}{\beta}}B\left(\frac{1+\alpha}{\beta},-\frac{1+\alpha+\gamma\beta}{\beta}\right),
    \end{align*}
    where $B(\cdot,\cdot)$ is the Beta function.
\end{lemma}
\begin{lemma}
\label{condition for convergence in L^{p}}
    Let $\{X_{n}\}_{n\in \mathbb{N}}$ be a sequence of random variables. Then $X_{n}$ converges in $L^{2p}(\Omega)$ for some $p\in \mathbb{N}$ if there exists some $r\in\R$ such that $\mathbb{E}[X_{n}^{2p-q}X_{m}^{q}]$ converges to $r$ as $m,n\to \infty$ for all $1\leq q\leq 2p$. 
    \begin{proof}
    Suppose that $\mathbb{E}[X_{n}^{p-q}X_{m}^{q}]$ converges to some $r\in \mathbb{R}$ as $m,n\to \infty$ for all $0\leq q\leq 2p$. Obviously it implies $X_{n}\in L^{2p}(\Omega)$. Then 
    \begin{align*}
        \mathbb{E}[|X_{n}-X_{m}|^{2p}]&=\mathbb{E}[(X_{n}-X_{m})^{2p}]\\
        &=\sum_{q=0}^{2p} {2p\choose q}(-1)^{q}\mathbb{E}[X_{n}^{q}X_{m}^{2p-q}].
    \end{align*}
    Letting $m,n$ converge to $\infty$, we get 
    \begin{align*}
        \lim_{n,m\to\infty}\mathbb{E}[|X_{n}-X_{m}|^{2p}]=r\sum_{q=1}^{2p}(-1)^{q}{2p\choose q}=r(1-1)^{2p}=0,
    \end{align*}
    which implies $X_{n}$ is a Cauchy sequence in $L^{2p}(\Omega)$. 
    \end{proof}
\end{lemma}
\begin{remark}
  The condition stated in \Cref{condition for convergence in L^{p}} is in fact both sufficient and necessary. However, since the necessity is not required in this article, we leave its verification to the interested reader.
\end{remark}
\begin{lemma}
\label{even L^{p}}
     If $X_{n}$ converges to $X$ in $L^{p}(\Omega
    )$ for all $1\leq p<\infty$, then $X_{n}$ converges to $X$ in $L^{q}(\Omega)$ for all $q<p$.
    \begin{proof}
        By Jensen's inequality, we have
        \begin{align*}
        \mathbb{E}[|X_{n}-X|^{q}]\leq\mathbb{E}[|X_{n}-X|^{p}]^{\frac{q}{p}},
        \end{align*}
        which implies that $X_{n}$ converges to $X$ in $L^{q}(\Omega)$ since $X_{n}$ converges to $X$ in $L^{p}(\Omega)$. 
    \end{proof}
\end{lemma}
\begin{lemma}
\label{conditional expectation extra independent sigma}
    Let $X$ be an integrable random variable on a probability space $(\Omega,\mathcal{F},\mathbb{P})$ and $\mathcal{G},\mathcal{H}\subset\mathcal{F}$ be two $\sigma$-algebras. Assume that $\sigma(\sigma(X)\cup\mathcal{H})$ is independent of $\mathcal{G}$, then we have
    \begin{align*}
        \mathbb{E}[X|\sigma(\mathcal{G\cup\mathcal{H}})]=\mathbb{E}[X|\mathcal{H}]
    \end{align*}
    almost surely.
    \begin{proof}
        It suffices to show that 
        \begin{itemize}
            \item $\mathbb{E}[X|\mathcal{H}]$ is $\sigma(\mathcal{G\cup\mathcal{H}})$ measurable,
            \item $\mathbb{E}[X|\mathcal{H}]$ is integrable,
            \item for all $A\in \sigma(\mathcal{G\cup\mathcal{H}}):$
            \begin{align*}
                \int_{A}\mathbb{E}[X|\mathcal{H}]d\mathbb{P}=\int_{A}Xd\mathbb{P}.
            \end{align*}
        \end{itemize}
        The first two are trivial as $\mathbb{E}[X|\mathcal{H}]$ is $\mathcal{H}$ measurable and $X$ is integrable. Note that 
        $\sigma(\mathcal{G}\cup\mathcal{H})=\sigma(\{E\cap F; E\in\mathcal{G},F\in\mathcal{H}\})$, it therefore suffices to show for all $E\in\mathcal{G},F\in\mathcal{H}$, we have
        \begin{align*}
            \int_{E\cap F}\mathbb{E}[X|\mathcal{H}]d\mathbb{P}=\int_{E\cap F}Xd\mathbb{P}.
        \end{align*}
        Since $\mathbb{E}[X|\mathcal{H}]1_{F}$ is $\mathcal{H}$ measurable and consequently is $\sigma(\sigma(X)\cup \mathcal{H})$ measurable, by independence we have
        \begin{align*}
            \int_{E\cap F}\mathbb{E}[X|\mathcal{H}]d\mathbb{P}&=\mathbb{E}[1_{F}1_{E}\mathbb{E}[X|\mathcal{H}]]\\
            &=\mathbb{E}[1_{E}]\mathbb{E}[1_{F}\mathbb{E}[X|\mathcal{H}]]\\
            &=\mathbb{E}[1_{E}]\mathbb{E}[1_{F}X].
        \end{align*}
        Since $1_{F}X$ is $\sigma(\sigma(X)\cup \mathcal{H})$ measurable, then by independence we have
        \begin{align*}
             \int_{E\cap F}\mathbb{E}[X|\mathcal{H}]d\mathbb{P}&=\mathbb{E}[1_{E}1_{F}X]\\
             &=\int_{E\cap F}Xd\mathbb{P}
        \end{align*}
        as required.
    \end{proof}
\end{lemma}
\begin{lemma}
\label{indicator representation}
    Let $X$ be a continuous random variable defined on a probability space $(\Omega,\mathcal{F},\mathbb{P})$. Then for any Borel set $A$, we have 
    \begin{align*}
        \int_{A}\int_{\R}\frac{1}{2\pi}e^{i(X-x)p}e^{-\frac{\epsilon p^{2}}{2}}dpdx \to 1_{A}(X)
    \end{align*}
    in $L^{n}(\Omega)$ as $\epsilon\to0$ for $1\leq n <\infty$. 
    \begin{proof}
        It is clear that $\int_{A}\int_{\R}\frac{1}{2\pi}e^{i(X-x)p}e^{-\frac{\epsilon p^{2}}{2}}dpdx$ and $1_{A}(X)$ are in $L^{n}(\Omega)$. Since the intervals generate Borel sets, it suffices to show that for any interval $(a,b)$,
        \begin{align*}
            \mathbb{E}\left[\left|\int_{a}^{b}\int_{\R}\frac{1}{2\pi}e^{i(X-x)p}e^{-\frac{\epsilon p^{2}}{2}}dpdx-1_{(a,b)}(X)\right|^{n}\right]
        \end{align*}
        converges to $0$ as $\epsilon\to 0$. By \Cref{condition for convergence in L^{p}} and \Cref{even L^{p}}, it is enough to show
        \begin{align*}
         \mathbb{E}[\left(\int_{a}^{b}\int_{\R}\frac{1}{2\pi}e^{i(X-x)p}e^{-\frac{\epsilon_{1} p^{2}}{2}}dpdx\right)^{n-m}\left(\int_{a}^{b}\int_{\R}\frac{1}{2\pi}e^{i(X-x)p}e^{-\frac{\epsilon_{2} p^{2}}{2}}dpdx\right)^{m}]\to \mathbb{E}[1_{(a,b)}(X)]
        \end{align*}
        for all even $n$ and $1\leq m\leq n$ as $\epsilon_{1},\epsilon_{2}\to 0$. Let $F(x)$ be the cumulative function induced by $\mathbb{P}(X<x)$ which is continuous in this case, then by Fubini's theorem, we have
        \begin{align*}
            &\mathbb{E}[\left(\int_{a}^{b}\int_{\R}\frac{1}{2\pi}e^{i(X-x)p}e^{-\frac{\epsilon_{1} p^{2}}{2}}dpdx\right)^{n-m}\left(\int_{a}^{b}\int_{\R}\frac{1}{2\pi}e^{i(X-x)p}e^{-\frac{\epsilon_{2} p^{2}}{2}}dpdx\right)^{m}]\\
            &=\int_{(a,b)^{n}}\int_{\R^{n}}\frac{1}{(2\pi)^{n}}\mathbb{E}[e^{i\sum_{j=1}^{n}(X-x_{j})p_{j}}]e^{-\frac{\epsilon_{1} \sum_{j=1}^{n-m}p_{j}^{2}+\epsilon\sum_{j=n-m+1}^{n}p_{j}^{2}}{2}}dpdx\\
            &=\int_{(a,b)^{n}}\int_{\R^{n}}\frac{1}{(2\pi)^{n}}\int_{\R}e^{i\sum_{j=1}^{n}(y-x_{j})p_{j}}e^{-\frac{\epsilon_{1} \sum_{j=1}^{n-m}p_{j}^{2}+\epsilon\sum_{j=n-m+1}^{n}p_{j}^{2}}{2}}dF(y)dpdx\\
            &=\int_{\R}\int_{(a,b)^{n}}\frac{1}{(\sqrt{2\pi\epsilon_{1}})^{n-m}}\frac{1}{(\sqrt{2\pi\epsilon_{2}})^{m}}e^{-\sum_{j=1}^{n-m}\frac{(y-x_{j})^{2}}{2\epsilon_{1}}}e^{-\sum_{j=1}^{m}\frac{(y-x_{j})^{2}}{2\epsilon_{2}}}dxdF(y)\\
            &=\int_{\R}f_{\epsilon_{1}}^{n-m}(y)f_{\epsilon_{2}}^{m}(y)dF(y),
        \end{align*}
        where $f_{\epsilon}(y)=\mathbb{P}(\frac{a-y}{\sqrt{\epsilon}}<Z<\frac{b-y}{\sqrt{\epsilon}})$ is the probability of a standard normal random variable $Z$ staying between $(\frac{a-y}{\sqrt{\epsilon}},\frac{b-y}{\sqrt{\epsilon}})$. Note that, when $y\in (a,b)$
        \begin{align*}
            \lim_{\epsilon\to 0}f_{\epsilon}(y)=\lim_{\epsilon\to0}\mathbb{P}(\frac{a-y}{\sqrt{\epsilon}}<Z<\frac{b-y}{\sqrt{\epsilon}})=1,
        \end{align*}
        when $y\in(-\infty,a)\cup(b,\infty)$,
        \begin{align*}
            \lim_{\epsilon\to0}f_{\epsilon}(y)=\lim_{\epsilon\to0}\mathbb{P}(\frac{a-y}{\sqrt{\epsilon}}<Z<\frac{b-y}{\sqrt{\epsilon}})=0.
        \end{align*}
        Since $f_{\epsilon}(y)$ is bounded by $1$, then dominated convergence theorem and the continuity of $F(y)$, we have
        \begin{align*}
            \lim_{\epsilon_{1}\to0\epsilon_{2}\to0}\int_{\R}f_{\epsilon_{1}}^{n-m}(y)f_{\epsilon_{2}}^{m}(y)dF(y)&=\lim_{\epsilon_{1}\to0\epsilon_{2}\to0}\int_{(a,b)}f_{\epsilon_{1}}^{n-m}(y)f_{\epsilon_{2}}^{m}(y)dF(y)\\
            &\quad+\lim_{\epsilon_{1}\to0\epsilon_{2}\to0}\int_{(-\infty,a)\cup(b,\infty)}f_{\epsilon_{1}}^{n-m}(y)f_{\epsilon_{2}}^{m}(y)dF(y)\\
            &=\int_{a}^{b}dF(y)=\mathbb{E}[1_{(a,b)}(X)]
        \end{align*}
        as required.
    \end{proof}
\end{lemma}
\begin{remark}
  We then denote its limit as
    \begin{align*}
          \lim_{\epsilon\to0}\int_{A}\int_{\R}\frac{1}{2\pi}e^{i(X-x)p}e^{-\frac{\epsilon p^{2}}{2}}dpdx:=\int_{A}\int_{\R}\frac{1}{2\pi}e^{i(X-x)}dpdx.
    \end{align*}
    The result can be extended to any bounded Borel measurable function. However, since this result is not required in this article, we leave its verification to the interested reader.
\end{remark}
\begin{lemma}
\label{uncorrelated means independent}
    Let $X$ and $Y$ be two uncorrelated jointly Gaussian random variables, then they are independent.
\end{lemma}
\begin{lemma}
\label{independence of a bunch of random variables}
    Let $Y$ and $\{X_{j}\}_{1\leq j\leq n}$ be continuous random variables defined on the same probability space $(\Omega,\mathcal{F},\mathbb{P})$.  
Then $\sigma(Y)$ is independent of $\sigma(X_{1},\dots,X_{n})$ if and only if $Y$ is independent of every linear combination of $X_{1},\dots,X_{n}$.

    \begin{proof}
        Suppose that $\sigma(Y)$ is independent of $\sigma(\sigma(X_{1})\cup\cdots\cup\sigma(X_{n}))$, $Y$ is of course independent of any combination of $X_{1},\dots,X_{n}$ as they are $\sigma(\sigma(X_{1})\cup\cdots\cup\sigma(X_{n}))$ measurable. Suppose that $Y$ is independent of any linear combination of $X_{1},\dots,X_{n}$. Note that 
        \begin{align*}
        \sigma(\sigma(X_{1})\cup\cdots\cup\sigma(X_{n}))=\sigma(\{\bigcap_{j=1}^{n}E_{j}; E_{j}\in \sigma(X_{j}), 1\leq j\leq n\}),
        \end{align*}
        and $\{X^{-1}_{j}(B); B\in \mathcal{B}(\R)\} =\sigma(X_{j})$. It suffices to show that for all Boreal sets $A$ and $\{B_{j}\}_{1\leq j\leq n}$, we have
        \begin{align*}
            \mathbb{E}[1_{\{Y\in A\}}\prod_{j=1}^{n}1_{\{X_{j}\in B_{j}\}}]=\mathbb{E}[1_{\{Y\in A\}}]\mathbb{E}[\prod_{j=1}^{n}1_{\{X_{j}\in B_{j}\}}].
        \end{align*}
        By Fubini's theorem and \Cref{indicator representation}, we have
        \begin{align*}
             \mathbb{E}[1_{\{Y\in A\}}\prod_{j=1}^{n}1_{\{X_{j}\in B_{j}\}}]=\frac{1}{(2\pi)^{n+1}}\int_{\R^{n+1}\times A\times B_{1}\times\cdots\times B_{n}}\mathbb{E}[e^{i(Y-y)p+\sum_{j=1}^{n}i(X_{j}-x_{j})q_{j}}]dydpdxdq.
             \end{align*}
             Since $Y$ is independent of all linear combinations of $X_{1},\dots,X_{n}$, we have
             \begin{align*}
                \mathbb{E}[e^{i(Y-y)p+\sum_{j=1}^{n}i(X_{j}-x_{j})q_{j}}]=\mathbb{E}[e^{i(Y-y)p}]\mathbb{E}[e^{\sum_{j=1}^{n}i(X_{j}-x_{j})q_{j}}].
             \end{align*}
             Consequently, we have
             \begin{align*}
                 \mathbb{E}[1_{\{Y\in A\}}\prod_{j=1}^{n}1_{\{X_{j}\in B_{j}\}}]&=\frac{1}{(2\pi)^{n+1}}\int_{\R^{n+1}\times A\times B_{1}\times\cdots\times B_{n}}\mathbb{E}[e^{i(Y-y)p}]\mathbb{E}[e^{\sum_{j=1}^{n}i(X_{j}-x_{j})q_{j}}]dpdydxdq\\
                 &=\frac{1}{2\pi}\int_{\R\times A}\mathbb{E}[e^{i(Y-y)p}]dpdy\times \frac{1}{(2\pi)^{n}}\int_{\R^{n}\times B_{1}\times\cdots\times B_{n}}\mathbb{E}[e^{\sum_{j=1}^{n}i(X_{j}-x_{j})q_{j}}]dxdq\\
                 &=\mathbb{E}[1_{\{Y\in A\}}]\mathbb{E}[\prod_{j=1}^{n}1_{\{X_{j}\in B_{j}\}}],
             \end{align*}
             where the last equality is also an application of Fubini's theorem.
    \end{proof}
\end{lemma}
\begin{lemma}
\label{Gaussian representation}
    Let $\{X_{i}\}_{1\leq i\leq n}$ be a tuple of jointly Gaussian random variables. For each $X_{j}$, there exists a tuple of real numbers $\{a_{i}\}_{1\leq i\neq j\leq n}$ such that
    \begin{align*}
        \mathbb{E}[X_{j}|X_{i},1\leq i\neq j\leq n]&=\mathbb{E}[X_{j}]+\sum_{i\neq j}a_{i}X_{i},\\
        \mathbb{E}[(X_{j}-\mathbb{E}[X_{j}]-\sum_{i\neq j}a_{i}X_{i})X_{k}]&=0, \forall k\neq j,\\
        \sigma(X_{j}-\mathbb{E}[X_{j}]-\sum_{i\neq j}a_{i}X_{i}) &\independent \sigma(\sigma(X_{1})\cup\cdots\cup\sigma(X_{n-1})).
    \end{align*}
    \begin{proof}
        Without loss of generality, we can assume that all the random variables are centred, and pick $X_{n}$ as a representative.  Let $\Sigma$ be the covariance matrix of $X_{1},\dots,X_{n-1}$ and denote
        \begin{align*}
            b=\left(\text{Cov}(X_{1},X_{n}),\dots,\text{Cov}(X_{n-1},X_{n})\right)^{\intercal}.
        \end{align*}
        Since $\Sigma$ is symmetric positive definite, we can write
        \begin{align*}
            a=\Sigma^{-1}b,
        \end{align*}
        and one can verify that 
        \begin{align*}
              \mathbb{E}[(X_{n}-\sum_{i=1 }^{n-1}a_{i}X_{i})X_{k}]=0, \forall 1\leq k\leq n-1.
        \end{align*}
        Consequently $X_{n}-\sum_{i=1 }^{n-1}a_{i}X_{i}$ is independent of any linear combination of $X_{1},\dots,X_{n-1}$ by \Cref{uncorrelated means independent}. Then by \Cref{independence of a bunch of random variables}, we have
        \begin{align*}
            \sigma(X_{n}-\sum_{i=1 }^{n-1}a_{i}X_{i})\independent \sigma(\sigma(X_{1})\cup\cdots\cup\sigma(X_{n-1})).
        \end{align*}
        Therefore,
        \begin{align*}
            \mathbb{E}[X_{n}-\sum_{i=1 }^{n-1}a_{i}X_{i}|X_{1},\dots,X_{n-1}]=\mathbb{E}[X_{n}-\sum_{i=1 }^{n-1}a_{i}X_{i}]=0,
        \end{align*}
        which implies 
        \begin{align*}
            \mathbb{E}[X_{n}|X_{1},\dots,X_{n-1}]=\sum_{i=1 }^{n-1}a_{i}X_{i}.
        \end{align*}
    \end{proof}
\end{lemma}
\begin{lemma}
\label{lemma:Gaussiandeterminant}
    Let $\{X_{i}\}_{1\leq i\leq n}$ be a tuple of jointly centred Gaussian random variables, and $A_{n}\in \R^{n\times n}$ is the associated covariance matrix. We have
    \begin{align}
        \det(A_{n})=\text{Var}(X_{\pi(1)})\text{Var}(X_{\pi(2)}|X_{\pi(1)})\cdots\text{Var}(X_{\pi(n)}|X_{\pi(1)},\dots,X_{\pi(n-1)}),
    \end{align}
    where $\pi$ is any permutation of $\{1,\dots,n\}$.
    \begin{proof}
        It suffices to show 
        \begin{align}
         \label{property}
            \det(A_{n})=\text{Var}(X_{1})\text{Var}(X_{2}|X_{2})\cdots\text{Var}(X_{n}|X_{1},\dots,X_{n-1}),
        \end{align}
        as we can shuffle the conditional variances by relabeling $X_{i}$.
        We will proceed by induction, assuming that $n=1$, the covariance matrix is a scalar that is the variance of $X_{1}$. Thus the base case is satisfied. Suppose $\cref{property}$ holds for $n$, the covariance matrix $A_{n+1}$ of $X_{1},\dots,X_{n+1}$ has a form as follows:
        \begin{align*}
            A_{n+1}=
            \begin{bmatrix}
                A_{n} & b\\
                b^{\intercal}& \text{Var}(X_{n+1})
            \end{bmatrix}
            ,
        \end{align*}
        where 
        \begin{align*}
            b^{\intercal}=(\text{Cov}(X_{n+1},X_{1}),\dots,\text{Cov}(X_{n+1},X_{n})).
        \end{align*}
        By the formula of determinant of a block matrix, we have
        \begin{align*}
            \det{A_{n+1}}=\det(A_{n})\det(\text{Var}(X_{n+1})-b^{\intercal}A^{-1}_{n}b).
        \end{align*}
        By \Cref{Gaussian representation}, there exists $a=(a_{1},..,a_{n})^{\intercal}$ such that
        \begin{align}
        \label{independent of gaussian representation}
            \mathbb{E}[X_{n+1}|X_{1},\dots,X_{n}]&=\sum_{i=1}^{n}a_{i}X_{i},\nonumber\\
            \mathbb{E}[(X_{n+1}-\sum_{i=1}^{n}a_{i}X_{i})X_{k}]&=0 \quad \forall 1\leq k\leq n,\nonumber\\
            \sigma(X_{j}-\sum_{i\neq j}a_{i}X_{i}) &\independent \sigma(\sigma(X_{1})\cup\cdots\cup\sigma(X_{n}))
        \end{align}
        This implies
        \begin{align}
        \label{independent}
        \mathbb{E}[(X_{n+1}-\sum_{i=1}^{n}a_{i}X_{i})\sum_{j=1}^{n}a_{j}X_{j}]=0,
        \end{align}
        and 
        \begin{align*}
            b_{k}&=\text{Cov}(X_{n+1},X_{k})\\
            &=\mathbb{E}[X_{n+1}X_{k}]\\
            &=\sum_{i=1}^{n}a_{i}\mathbb{E}[X_{i}X_{k}] \implies b=A_{n}a.
        \end{align*}
        Since $A_{n}$ is symmetric positive definite, we have
        \begin{align}
        \label{Variance}
        \begin{split}
        b^{\intercal}A_{n}^{-1}b&=b^{\intercal}a\\
        &=\sum_{i=1}^{n}a_{i}\text{Cov}(X_{n+1},X_{i})\\
        &=\mathbb{E}[X_{n+1}\sum_{i=1}^{n}a_{i}X_{i}].
        \end{split}
        \end{align}
        As such, we have
        \begin{align}
            \text{Var}(X_{n+1}|X_{1},\dots,X_{n})&=\mathbb{E}[(X_{n+1}-\mathbb{E}[X_{n+1}|X_{1},\dots,X_{n}])^{2}|X_{1},\dots,X_{n}]\nonumber\\
            &=\mathbb{E}[(X_{n+1}-\sum_{i=1}^{n}a_{i}X_{i})^{2}|X_{1},\dots,X_{n}]\nonumber\\
            &=\mathbb{E}[(X_{n+1}-\sum_{i=1}^{n}a_{i}X_{i})^{2}]\nonumber\\
            &=\mathbb{E}[X_{n+1}^{2}-X_{n+1}\sum_{i=1}^{n}a_{i}X_{i}]\nonumber\\
            &=\text{Var}(X_{n+1})-b^{\intercal}A_{n}^{-1}b\nonumber,
        \end{align}
        where we use \cref{independent of gaussian representation} in the third equality, \cref{independent} in the fourth inequality, and \cref{Variance} in the last equality. Therefore, 
        \begin{align*}
             \det{A_{n+1}}&=\det(A_{n})\det(\text{Var}(X_{n+1})-b^{\intercal}A^{-1}_{n}b)\\
             &=\text{Var}(X_{1})\text{Var}(X_{2}|X_{2})\cdots\text{Var}(X_{n}|X_{1},\dots,X_{n-1})\text{Var}(X_{n+1}|X_{1},\dots,X_{n})
        \end{align*}
        as required.
    \end{proof}
\end{lemma}
\bibliographystyle{apalike}
\bibliography{citation}
\end{document}